\documentclass[a4paper,12pt,oneside]{amsart}
\usepackage[left=3cm, right=3cm,top=3.5cm,bottom=3.5cm]{geometry}
\usepackage{algorithm, booktabs}
\usepackage{cite}
\usepackage{mathtools}
\usepackage{titlesec}
\titleformat{\section}
  {\centering \normalfont\fontsize{15}{15}\bfseries}{\thesection}{1em}{}
  \titleformat{\subsection} 
  {\normalfont\fontsize{13}{13}\bfseries}{\thesubsection}{1em}{}
\usepackage{enumitem}
\mathtoolsset{showonlyrefs}
\usepackage{amsfonts}
\usepackage{amsmath}
\usepackage{amssymb}
\usepackage{url}
\usepackage{faktor}
\usepackage{amscd}
\usepackage{array}
\usepackage{subfigure}
\usepackage{mathdots}
\usepackage{mathrsfs}
\usepackage{graphicx}
\usepackage{color}
\usepackage{pgf}
\usepackage{multirow}
\usepackage{tikz}
\usepackage{graphicx}

\newcommand{\N}{\ensuremath{\mathbb{N}}}

\renewcommand{\S}{\ensuremath{\mathbb{S}}}

\newcommand{\R}{\ensuremath{\mathbb{R}}}

\newcommand{\X}{\mathbb{X}}

\newcommand{\cO}{\mathcal{O}}
\newcommand{\cH}{\mathcal{H}}
\newcommand{\cG}{\mathcal{G}}
\newcommand{\sdd}{\S ^{d-1}}
\def\3{\ss}

\newmuskip\pFqmuskip
\newcommand*\pFq[6][8]{
  \begingroup 
  \pFqmuskip=#1mu\relax
  \begingroup\lccode`\~=`\,
  \lowercase{\endgroup\let~}\pFqcomma
  {}_{#2}F_{#3}{\left(\genfrac..{0pt}{}{#4}{#5};#6\right)}%
  \endgroup
}
\newcommand*\pRegFq[6][8]{
  \begingroup 
  \pFqmuskip=#1mu\relax
  \begingroup\lccode`\~=`\,
  \lowercase{\endgroup\let~}\pFqcomma
  {}_{#2}\tilde{F}_{#3}{\left(\genfrac..{0pt}{}{#4}{#5};#6\right)}%
  \endgroup
}
\newcommand{\pFqcomma}{\mskip\pFqmuskip}

\DeclareMathOperator*{\spann}{span}
\DeclareMathOperator{\dist}{dist}

\DeclareMathOperator*{\vol}{vol}

\newtheorem{thm}{Theorem}[section]
\newtheorem{lemma}[thm]{Lemma}

\newtheorem{definition}[thm]{Definition}
\newtheorem{example}[thm]{Example}
\newtheorem{corollary}[thm]{Corollary}
\newtheorem{proposition}[thm]{Proposition}
\usepackage{tcolorbox}
\tcbset{colback=gray!1,colframe=gray!20!white}

\begin{document}
\title[]{$t$-Design Curves and Mobile Sampling on the Sphere}
\author[M.~Ehler]{Martin Ehler}
\address[M.~Ehler]{University of Vienna,
Faculty of Mathematics, Vienna, Austria
}
\email{martin.ehler@univie.ac.at}
\author[K.~Gr\"ochenig]{Karlheinz Gr\"ochenig}
\address[K.~Gr\"ochenig]{University of Vienna,
Faculty of Mathematics, Vienna, Austria
}
\email{karlheinz.groechenig@univie.ac.at}

\thanks{K.\ G.\ was
  supported in part by the  project P31887-N32  of the
Austrian Science Fund (FWF)}

\begin{abstract}
In analogy to classical spherical $t$-design points, we introduce the
concept of $t$-design curves on the sphere. This means that the line
integral along a $t$-design curve integrates polynomials of degree $t$
exactly. For low degrees we construct explicit examples. We also derive lower asymptotic bounds on
the lengths of $t$-design curves. Our main results prove  the existence of
asymptotically optimal $t$-design curves in the Euclidean $2$-sphere
and the existence of $t$-design curves in the $d$-sphere. 
\end{abstract}

\maketitle

\section{Introduction}
Spherical designs are point sets in the sphere $\S ^d = \{x\in \R
^{d+1}: \|x\|=1\}$ that yield exact quadrature rules with constant weights for polynomial
spaces. Thus a finite set $X_t\subseteq \S ^d$
is a $t$-design (or $X_t$ consists of $t$-design points), if for every algebraic polynomial $f$ in $d+1$
variables of (total) 
degree $t$ one has
\begin{equation}
  \label{eq:c15b}
\frac{1}{|X_t|} \sum_{x\in X_t}  f(x)  = \int _{\S ^d} f \, .
 \end{equation}
This concept plays an important role in numerical analysis,
approximation theory, and many related fields, and the theory,
construction, and applications of spherical designs has become a
highly developed art. See~\cite{Delsarte:1977aa,Seidel:2001aa,Seymour:1984bh} and 
\cite{Brauchart:fk,Graf:2011lp,Sloan:2004qd,Womersley:2018we} for a
sample of  papers on  spherical $t$-design
points. In particular, 
 the existence of asymptotically optimal $t$-design points in the
 $d$-sphere has long been an open problem and has  eventually been proved by Bondarenko, Radchenko, and
Viazovska in \cite{Bondarenko:2011kx}.  

In this paper we study a variation where points are replaced by
curves. The goal is again to obtain    quadrature formulas along
curves in  the sphere   that are exact for polynomials of a given
degree. The central  notion  is the definition of a $t$-design curve.  
Precisely, a closed, piecewise smooth  curve $\gamma :
[0,1] \to \S ^d$ with at most finitely many self intersections and with arc length $\ell (\gamma )$ is called a
$t$-design curve in  $\S ^d$ if the line integral integrates exactly all algebraic polynomials in $d+1$
variables of degree $t$,
\begin{equation}
  \label{eq:int1}
\frac{1}{\ell (\gamma )} \int _\gamma f = \int _{\S ^d} f\,.
\end{equation}
The use of curves instead of point evaluations in
definition~\eqref{eq:int1}   is motivated  by
 numerous analogous  applications of curves for the  collection  and
 processing of  data on the sphere. Here is a short list of
 applications of curves in a similar spirit: Low-discrepancy curves were
discussed in \cite{Ramamoorthy:2008di} as an efficient coverage of space with applications in robotics. See also the textbook \cite{LaValle:2006aa} on robotics, where curves are derived for motion planning to obtain an
optimal path under several side constraints. Space-filling
curves are used as dimensionality reduction tools in optimization,
image processing, and deep learning
cf.~\cite{Goetzel:1999aa,Chen:2022aa,Tsinganos:2021aa}. Curves are applied in \cite{Ehler:2019aa} to approximate 
probability measures. The concept of principal curves is discussed in \cite{Hauberg:2015dw,Hastie:1989rw,Kegl:2000ap,Lee:2021hj} to best fit given data. In another statistical context, the information tuning curve quantifies discriminatory abilities of populations of neurons \cite{Ringach:2010aa,Kang04}. Motivated by more geometric questions, length and thickness of ropes on
spheres are studied in \cite{Gerlach:2011aa,Gerlach:2011ab} as variants of packing problems.  In the context of optimization problems the shortest closed space curve to inspect a sphere is determined in \cite{Ghomi:2021aa}, variations are discussed in \cite{Zalgaller:2005aa}. Energy minimization and geometric arrangements in biophysics lead to optimality questions of knots and ropes \cite{Cantarella:2002aa,OHara:1992aa,Yu:2021aa}. In mobile sampling curve trajectories provide sampling sets that enable efficient signal reconstruction 
\cite{Benedetto:2000aa,Jaming:2020xu,Grochenig:2015ya,Jaye:2022aa,Unnikrishnan:2013df,Unnikrishnan:2013aa,Rashkovskii:2020aa}. 

Our goal is to study integration  on the sphere
by using information along closed curves rather than point
evaluations.  The new notion of $t$-design curves in \eqref{eq:int1}  addresses exact integration on the
sphere along curves and the related problem of the exact
reconstruction of bandlimited functions on the sphere. The pertinent questions of $t$-design curves on spheres are similar to
those of spherical $t$-design points. \\
Problem $(A)$:  What is  the minimal (order of the) arc  length of a $t$-design curve? \\
Problem $(B)$:  Do  $t$-design curves exist  on $\S ^d$ for all $t\in\N$? \\
Problem $(C)$: If yes, are there  $t$-design curves on $\S ^d$
achieving  the optimal order of arc length? \\
Problem $(D)$: Provide explicit constructions of $t$-design curves.

The answers to the analogous questions for $t$-design points  
have a long history and 
 culminate in the solution of the
Korevaar-Meyers conjecture by Bondarenko, Radchenko, and
Viazovska~\cite{Bondarenko:2011kx} mentioned above. 

Our program is to make a first attempt at these questions for
$t$-design \emph{curves} on  $d$-spheres. We will offer answers to
$(A)$ and $(B)$ and give a solution of problem $(C)$ on the sphere $\S
^2$. As a contribution to Problem $(D)$ we will  construct some
examples of smooth $t$-design curves for small degrees $t$.  
\vspace{3mm}

\textbf{Results.}  In the following we
denote the space of algebraic polynomials of $d+1$ real
variables of (total) degree $t$ by
$\Pi _t$. As a necessary condition for the length of a
$t$-design curve we obtain the following answer to Problem $(A)$.

\begin{thm} \label{tint1}
  Assume that a piecewise smooth, closed curve $\gamma : [0,1] \to \S
  ^d$ satisfies
  $$
\frac{1}{\ell (\gamma )} \int _\gamma f = \int _{\S ^d} f \qquad
\text{ for all } \, f \in \Pi _t  \, .
$$
Then its length is bounded from below by 
$$
\ell (\gamma ) \geq C_d t^{d-1 }
$$
with some constant $C_d>0$ that may depend on the dimension $d$ but is independent of $t$ and $\gamma$.
\end{thm}
By comparison, 
a spherical $t$-design requires $|X_t| \asymp t^d$
points~ \cite{Harpe:2005fk,Seidel:2001aa,Delsarte:1977aa}\footnote{We write 
  $\lesssim$ if the left-hand-side is bounded by a constant times the
  right-hand-side. If  $\lesssim$ and $\gtrsim$ both hold, then we
  write $\asymp$.}.  

The next challenge is to prove the existence of $t$-design curves that
match the asymptotic order $ t^{d-1}$. For the unit sphere in $\R
^3$ we succeeded in proving the existence. This solves Problem $(C)$
for $\S ^2$. 

\begin{thm} \label{tint2}
  In $\S ^2$ there exists a sequence of $t$-design curves $\left(\gamma _t\right)_{t\in\N}$ with length $\ell (\gamma _t) \asymp t $. 
\end{thm}
Note that even for $d=2$ the corresponding problem of the  existence
of spherical $t$-designs  points was solved only in 2011
~\cite{Bondarenko:2011kx}. In our proof we will make substantial use
of the result from~\cite{Bondarenko:2011kx}. 
In dimension $d\geq 3$ we prove the existence of $t$-design
curves. This is a solution to Problem $(B)$. 

\begin{thm} \label{tint3}
In $\S^d$ for $d\geq 3$ there exists a sequence of $t$-design curves $\left(\gamma
_t\right)_{t\in\N}$, such that $\ell (\gamma _t) \lesssim t^{d(d-1)/2}$. 
\end{thm}
This asymptotic order does not match our lower bounds for $d\geq 3$ in Theorem~\ref{tint1}. In the analogous problem of $t$-design points, our Theorem~\ref{tint3} 
corresponds to the upper bounds of Korevaar and Meyers~\cite{korevaar93}
from 1993. It remains an interesting challenge to derive the existence
of $t$-design curves on $\S ^d$ that match the bounds of
Theorem~\ref{tint1}.

The existence theorems are constructive only in part, as they are
based on the non-constructive results of Bondarenko, Radchenko and
Viazovska~\cite{Bondarenko:2011kx}. However, when starting with  a spherical
$t$-design, our construction  shows that every $t$-design
curve may be chosen to be a closed, piecewise smooth curve that
consists of arcs of Euclidean circles  (by a circle in $\S
^d$ we mean a circle in  the intersection of a  $2$-dimensional
subspace of $\R ^{d+1}$ with $\S ^d$).

As a small contribution to Problem $(D)$  we will discuss some  explicit
constructions of smooth $t$-design curves for very low polynomial  degrees
($t=1,2,3$). Explicit constructions of $t$-design points and curves
remain a difficult problem with many open threads. 

\vspace{3mm}

\textbf{Mobile sampling.} Mobile sampling refers to the approximation
or reconstruction of a function from its values along a
curve~\cite{Unnikrishnan:2013aa,Unnikrishnan:2013df}. The rationale
for this mode of data acquisition is the  small number of required
sensors. Sampling a function  along a curve requires only one sensor, 
whereas the 
sampling at a point set, e.g., $t$-design points, requires many
sensors. In engineering applications, it is   natural to
assume that the function $f$  to be sampled is 
bandlimited on $\R ^d$, i.e., the support of  the Fourier transform
$\hat{f}$ is compact.

Transferred to the sphere $\S ^d$, a function on the
sphere is bandlimited, if it is a polynomial restricted to the sphere. Its degree is a measure for the bandwidth. A typical and natural scenario for mobile sampling on the
sphere would be the surveillance of meteorological or geophysical data
along airplane routes. The goal would be to reconstruct the complete
data globally, which means literally on the entire ``globe'', i.e.,
$\S ^2$. The connection between  $t$-design curves and mobile sampling
on the sphere is explained in the
following statement.

\begin{corollary} \label{corintro}
 Let $\gamma $ be a $2t$-design curve on $\S ^d$ and $f$ a polynomial
 of degree $t$. Then
  \begin{equation}
   \label{eq:int4}
   \frac{1}{\ell (\gamma )} \int _\gamma |f|^2 = \int _{\S ^d} |f|^2
   \, .
 \end{equation}
Furthermore,  $f$ is uniquely determined by its
 values along  $\gamma $. 
\end{corollary}
Clearly, \eqref{eq:int4} follows immediately from the assumption,
because  $f \in \Pi _t$ implies that  $|f|^2 \in \Pi _{2t}$. The
uniqueness and 
an explicit  reconstruction formula will be derived in
Section~\ref{sec:Rd}.

We also discuss some elementary consequences of $t$-design curves on the sphere in
high-dimensional Euclidean quadrature. Generalized Gauss-Laguerre
quadrature combined with spherical $t$-design curves lead to exact integration of 
polynomials of total degree $t$ with respect to the measure
$\mathrm{e}^{-\|x\|}\mathrm{d}x$ on $\R^d$.

\vspace{3mm}

\textbf{Methods.}
Both Theorems~\ref{tint1} and \ref{tint2}  are based on the
existence of optimal $t$-design points. An immediate guess would be to
  connect  $t$-design points  along
geodesic arcs in $\S ^d$ and hope that a suitable order of points
would yield a $t$-design curve.  As there are $O (t^d)$ points in a
$t$-design with a distance $O (t^{-1})$ to the nearest neighbor, the
solution of the traveling salesman problem would lead to a curve of
the desired length $O (t^{d-1})$~\cite[Lemma 3]{Ehler:2019aa}. However, so far this
idea has not been fruitful,  and we do not know how such a path would
yield  exact quadrature.

Our idea is to connect the point evaluation $f\to f(x), x\in \S ^d$ to
an integral over the boundary of a spherical cap by means of a formula
of Samko~\cite{Samko:1983oa}. In $\S ^2$ the boundary of a spherical cap is a
circle and thus the union of such circles with centers at $t$-design
points yields a first quadrature rule. To generate a single closed
curve from this union of circles, we invoke some combinatorial arguments from graph
theory, such as spanning trees and Eulerian  paths.  The extension to higher dimensions is by
induction on the dimension $d$. Here we need some additional
properties from spherical geometry. 

\vspace{3mm}

\vspace{3mm}

\textbf{Outlook.} 
Spherical $t$-design points have proved a
rich field of research with deep mathematical questions. The new theory of $t$-design curves offers a similarly rich playground
for both challenging mathematics and for the investigation of
associated numerical and computational questions and applications.

On the mathematical side the most immediate question is the existence
of $t$-design curves of asymptotically optimal length in $\S^d$ for $d\geq 3$. Another
direction is the exploration of $t$-design curves on general compact
Riemannian manifolds (extending the work on $t$-design points in
\cite{Ehler:2020aa,Etayo:2016qq,Gariboldi:eu}). Theorem \ref{tint1} on
the lower bound on the length of a $t$-design curve carries over to
the manifold setting, but all constructive aspects are wide open. 

Next one might want to impose additional conditions on the curves. 
Our construction yields piecewise smooth, closed curves with  
finitely many corners and  self-intersections. For aesthetical reasons
one might want $t$-design curves to be smooth and simple, i.e.,
without corners and self-intersections. So far we know such examples
only for degrees $t\leq 3 $ on $\S ^2$.  Other aspects to be
considered might be curvature,  contractibility (or other homotopy constraints),
or  the ratio between inner and outer area for closed curves on
surfaces. At this time we are far from  understanding  any of these
questions.

  \bigskip
 The outline is as follows: In Section \ref{sec:T-design}, we introduce the concept of $t$-design curves in $\S^d$, and we derive asymptotic lower bounds on the curves' length. Smooth spherical $t$-design curves in $\S^2$ for $t=1,2,3$ are provided in Section \ref{sec:123}. Section \ref{sec:general preparation subspheres} is dedicated to some preparations for our two main theorems. The first one on the existence of asymptotically optimal $t$-design curves in $\S^2$ is derived in Section \ref{sec:sphere}. 
 The existence of $t$-design curves in $\S^d$ is proved in Section
 \ref{sec:sphere II}. In Section \ref{sec:Rd}, we briefly discuss the
 use of $t$-design curves in mobile sampling and for
 high-dimensional quadrature on $\R^d$. 

\section{From points to curves}\label{sec:T-design}
We denote the collection of classical polynomials of total degree at
most $t\in\N$ in $d+1$ variables on $\R^{d+1}$ by $\Pi_t$. Each $f\in
\Pi_t$ can be evaluated on  the unit sphere
$\S^d=\{x\in\R^{d+1}:\|x\|=1\}$. We always use the normalized surface measure, so that $\int_{\S^d} 1 = 1$. The (rotation-) invariant metric on $\S ^d$ is 
\begin{equation}\label{eq:dist def}
\dist(x,y) = \arccos \langle x,y \rangle\,.
\end{equation}
It measures the length of the geodesic arc connecting $x$ and $y$ on $\S^d$. 

\subsection{$t$-design points}
For $t\in\N$, a finite set $X\subset\S^d$ is called \emph{$t$-design
  points} (or simply $t$-design) in $\S^d$ if
\begin{equation}\label{eq:t d def}
        \frac{1}{|X|}\sum_{x\in X} f(x) = 	\int_{\S^d}
        f ,\qquad \forall f\in\Pi_t.
\end{equation}
We call $(X_t)_{t\in\N}$ a \emph{sequence of $t$-design points} in $\S^d$ if each $X_t$ is a $t$-design for $t\in\N$. 

It turns out that each sequence $(X_t)_{t\in\N}$ of $t$-design points in $\S^d$ must satisfy 
\begin{equation}\label{eq:lo b t d}
	|X_t|\gtrsim  t^d,\qquad t\in\N,
\end{equation}
where the constant may depend on $d$ but is independent of $t$, 
cf.~\cite{Breger:2016rc}. A sequence of $t$-design points $(X_t)_{t\in\N}$ in $\S^d$ is called \emph{asymptotically optimal} if 
\begin{equation*}
	|X_t|\asymp t^d, \qquad t\in\N.
\end{equation*}
Again, we allow the constant to depend on $d$. 
Asymptotically optimal point sequences do exist \cite{Bondarenko:2011kx,Ehler:2020aa,Etayo:2016qq,Gariboldi:eu}.

\subsection{$t$-design curves}\label{sec:T-design curves}
We now introduce a new concept by switching from points to a curve. By
a curve we mean a  continuous, piecewise differentiable  function
$\gamma:[0,1]\rightarrow \S^d$ with at most finitely many
self-intersections. Since the sphere is a closed manifold, we only consider closed curves.  

We may interpret $\gamma$ as a space curve in $\R^{d+1}$, so that its length is 
$$
\ell(\gamma)=\int_0^1 \|\dot{\gamma}(s)\|\mathrm{d}s,
$$
where the speed   $\|\dot{\gamma}\|$ of the curve  is defined almost
everywhere. The line integral is   
\begin{equation*}
	\int_{\gamma} f = \int_0^1 f(\gamma (s))\|\dot{\gamma}(s)\|\mathrm{d}s,
\end{equation*}
so that $\ell(\gamma)=\int_\gamma 1$. Note that $\int _\gamma f $ does
not depend on the parametrization and orientation of the curve.

In analogy  to \eqref{eq:t d def}, we now introduce $t$-design curves. 
\begin{definition}\label{def:design curve}
	For $t\in\N$, we say that $\gamma$ is a \emph{$t$-design curve} in $\S^d$ if 
	\begin{equation} \label{mhmh}
	\frac{1}{\ell(\gamma)}\int_{\gamma} f = \int_{\S^d}
        f ,\qquad f\in\Pi_t. 
	\end{equation}
 A sequence of curves $(\gamma_t)_{t\in\N}$ is called a sequence 
	of \emph{$t$-design curves} in $\S^d$ if each $\gamma_t$ is a $t$-design for $t\in\N$. 
\end{definition}
Analogously to \eqref{eq:lo b t d}, one now expects lower asymptotic bounds on $\ell(\gamma_t)$, see also \cite[Theorem 3 in Section 5]{Ehler:2019aa}. The following is Theorem~\ref{tint1} of the Introduction.
\begin{thm}\label{thm:t-design length}
If $(\gamma_t)_{t\in\N}$ a sequence of $t$-design curves in $\S^d$, then  
	\begin{equation}\label{eq:asympt lower bound curve}
		\ell(\gamma_t) \gtrsim t^{d-1},\quad t\in\N.
	\end{equation}
\end{thm}
\begin{proof}
  We use some results from ~\cite{Breger:2016rc}
  and~\cite{Brandolini:2014oz}. Let $\Gamma_t=\gamma_t([0,1])$ be the trajectory of $\gamma_t$. The covering radius $\rho _t$ of $\Gamma_t$ is defined as  
	\begin{equation}\label{eq:cov rad}
		\rho_t:=\sup_{x\in\S^d}\left(\inf_{y\in\Gamma_t}\dist(x,y)\right).
	\end{equation}
By this definition there is $x\in\S^d$ such that the closed ball $B_{\rho_t/2}(x)=\{y\in\S^d:\dist(x,y)\leq \frac{\rho_t}{2}\}$ of radius $\rho_t/2$ centered at $x$ does not intersect $\Gamma_t$, i.e., 
\begin{equation}\label{eq:00987}
B_{\rho_t/2}(x)\cap \Gamma_t = \emptyset.
\end{equation}
Let us denote the Laplace-Beltrami operator on the sphere $\S^d$ by $\Delta$ and the identity operator by $I$. 
According to \cite[Lemma 5.2 with $s=2d$ and $p=1$]{Breger:2016rc},
see also \cite{Brandolini:2014oz} for the original idea, there is a
function $f_t$ supported on  $B_{\rho_t/2}(x)$ such that 
\begin{equation}\label{eq:123456}
\|(I-\Delta)^d f_t\|_{L^1(\S^d)} \lesssim 1,\qquad \text{ and } \qquad  \rho_t^{2d}\lesssim \int_{\S^d} f_t.
\end{equation}
Since \eqref{eq:00987} implies $\int_{\gamma_t} f_t = 0$, the
$t$-design assumption and  \cite[Theorem 2.12]{Brandolini:2014oz} lead to
\begin{equation}\label{eq:654321}
\left|\int_{\S^d} f_t\right|
\lesssim t^{-2d} \|(I-\Delta)^d f_t\|_{L^1} . 
\end{equation}
By combining \eqref{eq:654321} with \eqref{eq:123456}, we deduce $\rho_t^{2d}\lesssim  t^{-2d}$, so that 
	\begin{equation}\label{eq:rho t etc}
		\rho_t\lesssim  t^{-1}.
	\end{equation}

To relate $\rho_t$ with $\ell(\gamma_t)$, we apply a packing
        argument.  Let $n$ be the maximum number of disjoint balls of
        radius $2\rho _t$ in $\S^d$, i.e., $B_{2\rho _t}(x_j) \cap
        B_{2\rho _t}(x_k) = \emptyset $ for $j,k= 1, \dots ,
        n$ with $j\neq k$. Then the balls $B_{4\rho _t}(x_j)$ cover $\S^d$, otherwise
        there is $x\in \S^d$, such that $x\not \in \bigcup _{j=1}^n
        B_{4\rho _t}(x_j)$ and $B_{2\rho _t}(x)$ is disjoint from all
        $B_{2\rho _t}(x_j)$ contradicting the maximality of $n$. 

        We note that every  ball $B_r(x)$ in $\S^d $ 
	with  radius $0<r\leq 1$
        has volume  $\vol(B_r(x))\asymp
        r^d$. Consequently, 
        \begin{equation*}
        1  \leq \sum _{j=1}^n
        \vol(B_{4\rho _t}(x_j)) \lesssim n \rho _t ^d,
                \end{equation*}
        so that we obtain
	\begin{equation*}
		n\gtrsim \rho_t^{-d},
	\end{equation*}
 This is known as the  Gilbert-Varshamov bound in coding theory,
 cf.~\cite{Barg:2002aa}.

Since $B_{2\rho _t}(x_j) \cap
        B_{2\rho _t}(x_k) = \emptyset $,  the distance between two distinct balls
 $B_{\rho_t}(x_j)$ and $B_{\rho_t}(x_k)$ is at least $2\rho_t$.
Due  to the definition of the covering radius $\rho_t$ and the compactness of $\Gamma_t$ and $\S^d$, the
 trajectory $\Gamma_t$  intersects each ball $B_{\rho_t}(x_k)$.  
 Therefore, the length of $\gamma _t$  must satisfy
	\begin{equation}\label{eq:rho 00}
		\ell(\gamma_t) \gtrsim n \rho_t \gtrsim \rho_t^{1-d}.
	\end{equation}

      Since \eqref{eq:rho t etc} is equivalent to $\rho^{-1}_t \gtrsim  t$, the bound \eqref{eq:rho 00} leads to
	\begin{equation*}
		\ell(\gamma_t) \gtrsim \rho_t^{1-d} \gtrsim t^{d-1}.\qedhere
	\end{equation*}
      \end{proof}
The lower bound on the  length of $\gamma _t$  leads to the concept of asymptotic optimality for curves.
\begin{definition}\label{def:as curve}
	A sequence $(\gamma_t)_{t\in\N}$ of $t$-design curves in $\S^d$ is called \emph{asymptotically optimal} if 
	\begin{equation*}
		\ell(\gamma_t)\asymp t^{d-1}, \qquad t\in\N.
	\end{equation*}
\end{definition}

In the remainder of the paper we study the existence of $t$-design  curves on the
sphere $\S ^d$.

\section{Some spherical $t$-design curves for small $t$ in $\S^2$}\label{sec:123}
Here we construct smooth $t$-design curves in $\S^2$ for $t=1,2,3$. For $1\leq k\in \N$ and $a\in [0,1]$, consider the family of curves $\gamma^{(k,a)}:[0,1]\rightarrow\S^2$ given by 
\begin{equation}\label{eq:gamma curve 0}
	\gamma^{(k,a)}(s):=
		\begin{pmatrix}
			a\cos(2\pi s)+(1-a)\cos(2\pi (2k-1) s)\\
			a\sin(2\pi s)-(1-a)\sin(2\pi (2k-1) s)\\
			2\sqrt{a(1-a)}\sin(2\pi k s)
		\end{pmatrix}.
\end{equation}
If $k=1$, then $\gamma^{(k,a)}$ describes a great circle, and it is easy to see that every great circle in $\S^d$ yields a $1$-design. 
The family $\gamma^{(k,a)}$ also gives rise to spherical $2$-design and $3$-design curves:
\begin{proposition}\label{prop:curve example}
	The curves  $\gamma^{(k,a)}$ have the following properties: 
	\begin{itemize}
		\item[(i)] $\gamma^{(k,a)}$ is a $1$-design
                  curve for all $k\in \N $. 
		\item[(ii)] There is $a_2\in(\frac{1}{2},1)$ such that $\gamma^{(2,a_2)}$ is a $2$-design curve.
		\item[(iii)] For $k\geq 3$, there is $a_k\in(\frac{1}{2},1)$ such that $\gamma^{(k,a_k)}$ is a $3$-design curve.
	\end{itemize}
\end{proposition}
To verify Proposition \ref{prop:curve example}, recall that the surface  measure on $\S^2$ is normalized, so that  $\int_{\S^2} 1
=1$. Due to the sphere's symmetries, integrals over the sphere of every monomial of odd degree vanish. Moreover, when $x,y,z$ denote the coordinate functions in $\R^3$, we have 
\begin{align}
	0 &=\int_{\S^2}xy=\int_{\S^2}xz=\int_{\S^2}yz,\label{eq:degree 1}\\
	\frac{1}{3} &= \int_{\S^2}x^2 = \int_{\S^2}y^2 = \int_{\S^2}z^2.\label{eq:1/3 xy}
\end{align}
By definition of length, we directly observe 
\begin{equation}\label{eq:const function is easy}
	\int_{\S^2} 1 =1= \frac{1}{\ell(\gamma^{(k,a)})}\int_0^1\|\dot{\gamma}^{(k,a)}(s)\| \mathrm{d}s =  \frac{1}{\ell(\gamma^{(k,a)})}\int_{\gamma^{(k,a)}} 1.
\end{equation}
To treat line integrals of the other monomials, will use the following lemma.  
\begin{lemma}\label{lemma:derivative of gamma}
	There are  real-valued coefficients $(c^{(k,a)}_n)_{n\in\N}\in
        \ell ^1(\mathbb{Z})$ such that
	\begin{equation*}
		\|\dot{\gamma}^{(k,a)}(s)\| = \sum_{n\in\N} c^{(k,a)}_n \cos(4\pi  k n s),
	\end{equation*}
	and $\gamma^{(k,a)}$ has an arc length parametrization. 
\end{lemma}
\begin{proof}[Proof of Lemma \ref{lemma:derivative of gamma}]
The arc length parametrization exists whenever  $\|\dot{\gamma}^{(k,a)}(s)\|$
is positive. Indeed, an elementary  calculation reveals that  
\begin{equation}\label{eq:uni}
	\|\dot{\gamma}^{(k,a)}(s)\|^2 = \alpha^{(k,a)}+\beta^{(k,a)}\cos(4\pi  k s)
\end{equation}
with the  nonnegative constants 
\begin{align}
	\alpha^{(k,a)} & = 4\pi^2\left((2k-1)^2 -2a(3k^2-4k+1) + 2a^2(k-1)^2\right) \label{eq:alpha and beta 1} \\
    & = 4\pi^2\left(a^2 + (2k-1)^2 (1-a)^2 + 2k^2
      a(1-a)\right) \label{eq:alpha and beta 1a} \\	\beta^{(k,a)}
                       & = 8\pi^2 a(1-a)(k-1)^2.\label{eq:alpha and
                         beta 2} 
\end{align}
Their difference 
\begin{equation}\label{eq:ab}
	\alpha^{(k,a)}-\beta^{(k,a)}=4\pi^2\Big(a^2 + (2k-1)^2 (1-a)^2
        + 2a(1-a) (2k-1)\Big) 
\end{equation}
is positive for all $a\in[0,1]$, so that also
$\|\dot{\gamma}^{(k,a)}(s)\|^2$ is positive for all $s\in\R$.  The
theorem of Wiener L\'evy implies that $s \mapsto
\|\dot{\gamma}^{(k,a)}(s)\|$ possesses an absolutely convergent
Fourier series. 
Since   $\|\dot{\gamma}^{(k,a)}(s)\|^2$ has period
$1/(2k)$, its square root has the same period, and thus 
only terms of the form $\cos(4\pi  k n s)$ appear in the Fourier
series of $\|\dot{\gamma}^{(k,a)}(s)\|$.  
\end{proof}

\begin{proof}[Proof of Proposition \ref{prop:curve example}]
(i) Exact integration of the constant function has already been checked in \eqref{eq:const function is easy}. 
Since $\gamma^{(k,a)}$ does not contain any term of the form $\cos(4 \pi k n s)$ for $n\in\N$, Lemma \ref{lemma:derivative of gamma} and the orthogonality relations of the Fourier basis imply
\begin{equation*}
	 \int_0^1\gamma^{(k,a)}(s)\|\dot{\gamma}^{(k,a)}(s)\|\mathrm{d}s
         = 0,
\end{equation*}
which matches the requirement that integrals of degree $1$ monomials vanish. 

(ii) 
For $k\geq 2$, by the definition of $\gamma ^{(k,a)}$ its  $x$ and
$y$-coordinates are trigonometric polynomials of degree $\leq 2k-1$
and the $z$-coordinate is a  trigonometric polynomial of degree $k$,
consequently the product of two distinct components of
$\gamma^{(k,a)}$  is a trigonometric polynomial of degree at most
$4k-2$ without a constant term. Therefore, the products $xy, xz,yz$ 
do not contain any term of the form $\cos(4 \pi n k t)$, for $n\in\N$, and we deduce 
\begin{equation*}
	0 = \int_{\gamma^{(k,a)}} x y =\int_{\gamma^{(k,a)}} x z =\int_{\gamma^{(k,a)}} y z,
\end{equation*}
which matches the identities \eqref{eq:degree 1} for the monomials of degree $2$.   

One sees directly from the definition of the coordinates of
$\gamma ^{(k,a)}$ that 
\begin{equation*}
	\int_{\gamma^{(k,a)}}x^2 = \int_{\gamma^{(k,a)}}y^2. 
\end{equation*}
Let $c^{(k,a)}_0=\ell(\gamma^{(k,a)})$ and $c^{(k,a)}_1$ be the
coefficients in Lemma \ref{lemma:derivative of gamma}. We now need to investigate
\begin{equation}\label{eq:aim rte}
	\frac{1}{\ell(\gamma^{(k,a)})} \int_{\gamma^{(k,a)}} z^2 =
        2a(1-a)\left(
          1-\frac{\frac{1}{2}c^{(k,a)}_1}{c^{(k,a)}_0}\right) =: \eta
        (a) \,.
\end{equation}
 We aim for a
parameter 
$a_k$ such that \eqref{eq:aim rte} equals $\frac{1}{3}$, but we cannot
solve this directly. For $a=0$ and $a=1$, the expression vanishes. We
will  verify that $\eta (1/2) > 1/3$. Then the continuity in $a$ and the intermediate value theorem ensure that
there is $a_k\in(\frac{1}{2},1)$ such that $\eta (a_k) = 1/3$. 

Rewriting \eqref{eq:aim rte} for $a=\frac{1}{2}$, we note that $\eta
(\frac{1}{2}) \geq \frac{1}{3}$, 
if and only if  
\begin{equation}\label{eq:nom and denom}
	\frac{\frac{1}{2}c^{(k,\frac{1}{2})}_1}{c^{(k,\frac{1}{2})}_0}\leq \frac{1}{3}\,.
\end{equation}
To obtain a lower bound on $c^{(k,\frac{1}{2})}_0$, we observe that by
\eqref{eq:ab} 
$\alpha^{(k,\frac{1}{2})} -\beta^{(k,\frac{1}{2})} = 4\pi^2 k^2$. Therefore, we have
\begin{equation*}
	\sqrt{\alpha^{(k,\frac{1}{2})}+\beta^{(k,\frac{1}{2})}\cos(4\pi
          k s)} \geq \sqrt{\alpha^{(k,\frac{1}{2})}
          -\beta^{(k,\frac{1}{2})}} \geq 2\pi k,
\end{equation*}
which leads to 
\begin{equation*}
	c^{(k,\frac{1}{2})}_0 = \int_0^1\sqrt{\alpha^{(k,\frac{1}{2})}+\beta^{(k,\frac{1}{2})}\cos(4\pi k s)}\mathrm{d}s \geq 2\pi k\,.
\end{equation*}
To obtain an upper bound on $\frac{1}{2}c^{(k,\frac{1}{2})}_1$, we
first observe that the substitution $s'=2 k s $ and periodicity lead to
\begin{align*}
	\frac{1}{2}c^{(k,\frac{1}{2})}_1 &= \int_0^1 \sqrt{\alpha^{(k,\frac{1}{2})}+\beta^{(k,\frac{1}{2})}\cos(4\pi k s)}
	\cos(4\pi k s)\mathrm{d}s\\
	& = \int_0^1 \sqrt{\alpha^{(k,\frac{1}{2})}+\beta^{(k,\frac{1}{2})}\cos(2\pi  s)}\cos(2\pi  s)\mathrm{d}s\,.
\end{align*}
We bound the positive part $\int_{-\frac{1}{4}}^{\frac{1}{4}}\ldots$ and the negative part $\int_{\frac{1}{4}}^{\frac{3}{4}}\ldots$ separately. The observation $\alpha^{(k,\frac{1}{2})} +\beta^{(k,\frac{1}{2})}=4\pi^2(2k^2-2k+1)\leq 8\pi^2k^2$ leads to
\begin{align*}
	\int_{-\frac{1}{4}}^{\frac{1}{4}}\sqrt{\alpha^{(k,\frac{1}{2})}+\beta^{(k,\frac{1}{2})}\cos(2\pi  s)}\cos(2\pi  s)\mathrm{d}s 
	& \leq \int_{-\frac{1}{4}}^{\frac{1}{4}} \sqrt{8}\pi k \cos(2\pi s)\mathrm{d}s =\sqrt{8}k\,.
\end{align*}
For the negative part, we recall $\alpha^{(k,\frac{1}{2})}
+\beta^{(k,\frac{1}{2})} \cos (2\pi s) \geq \alpha^{(k,\frac{1}{2})} -\beta^{(k,\frac{1}{2})} = 4\pi^2 k^2$ and obtain
\begin{equation*}
	\int_{\frac{1}{4}}^{\frac{3}{4}} \sqrt{\alpha^{(k,\frac{1}{2})}+\beta^{(k,\frac{1}{2})}\cos(2\pi  s)}\cos(2\pi  s)\mathrm{d}s 
	\leq \int_{\frac{1}{4}}^{\frac{3}{4}} 2\pi k \cos(2\pi s)\mathrm{d}s =-2k\,.
\end{equation*}
Combining these  bounds, we derive
\begin{equation*}
	\frac{1}{2}c^{(k,\frac{1}{2})}_1 \leq (\sqrt{8}-2)k \leq k.
\end{equation*}
Therefore, we do have 
\begin{equation*}
	\frac{\frac{1}{2}c^{(k,\frac{1}{2})}_1}{c^{(k,\frac{1}{2})}_0} \leq \frac{1}{2\pi}<\frac{1}{3}\,.
\end{equation*}
The intermediate value theorem ensures that we can match \eqref{eq:1/3 xy}.

(iii) The integrals over $\S^2$ of the
monomials of degree $3$ vanish. For $k\geq 3$, a product of $3$ components of $\gamma^{(k,a)}$ does not contain any terms of the form $\cos(4\pi k n s)$, where $n\in\N$.  Lemma \ref{lemma:derivative of gamma} implies that the integral over $\gamma^{(k,a)}$ of every degree $3$ monomial vanishes. 
\end{proof}

\begin{example}\label{ex:3t}
The values $a_2 \approx 0.7778$ and $a_3\approx 0.7660$ lead to the trajectories $\Gamma^{(2,a_2)}$ and $\Gamma^{(3,a_3)}$ depicted in Figure \ref{fig:approx 3 design}. The lengths satisfy $\ell(\gamma^{(2,a_2)})\approx 9.786$, and $\ell(\gamma^{(3,a_3)})\approx 14.232$.
\end{example}
\begin{figure}
\subfigure[$\gamma^{(2,a_2)}$]{
\includegraphics[width=.3\textwidth]{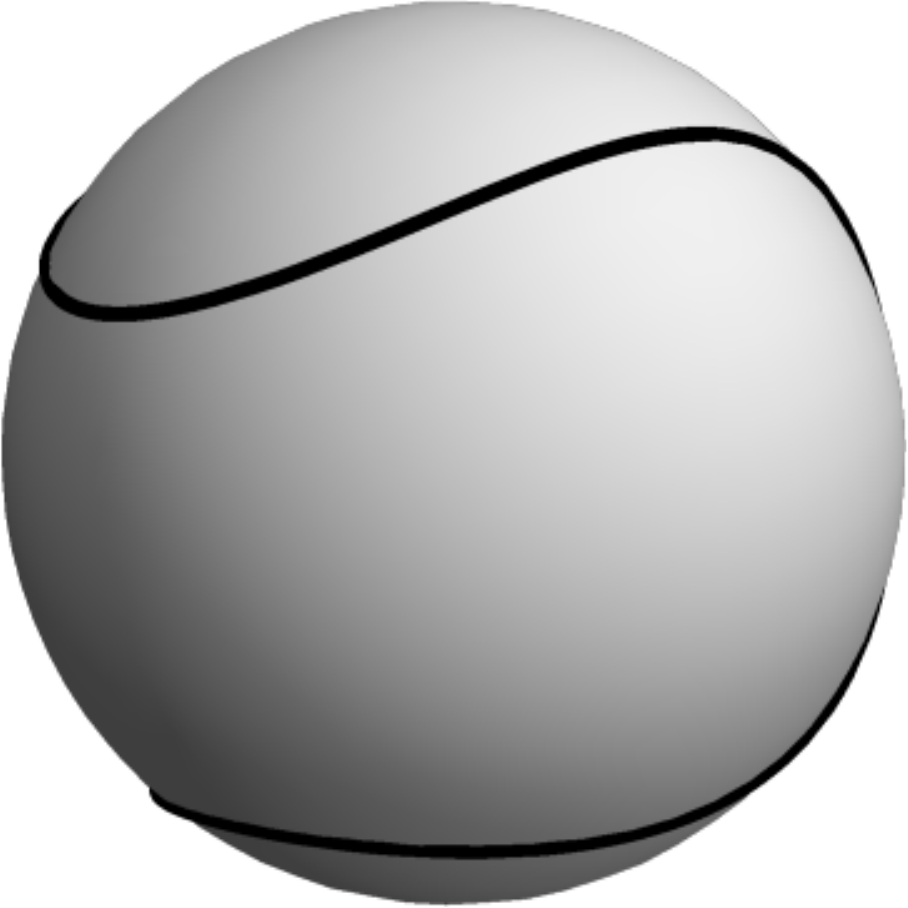}}\hspace{6ex}
\subfigure[$\gamma^{(3,a_3)}$]{
\includegraphics[width=.3\textwidth]{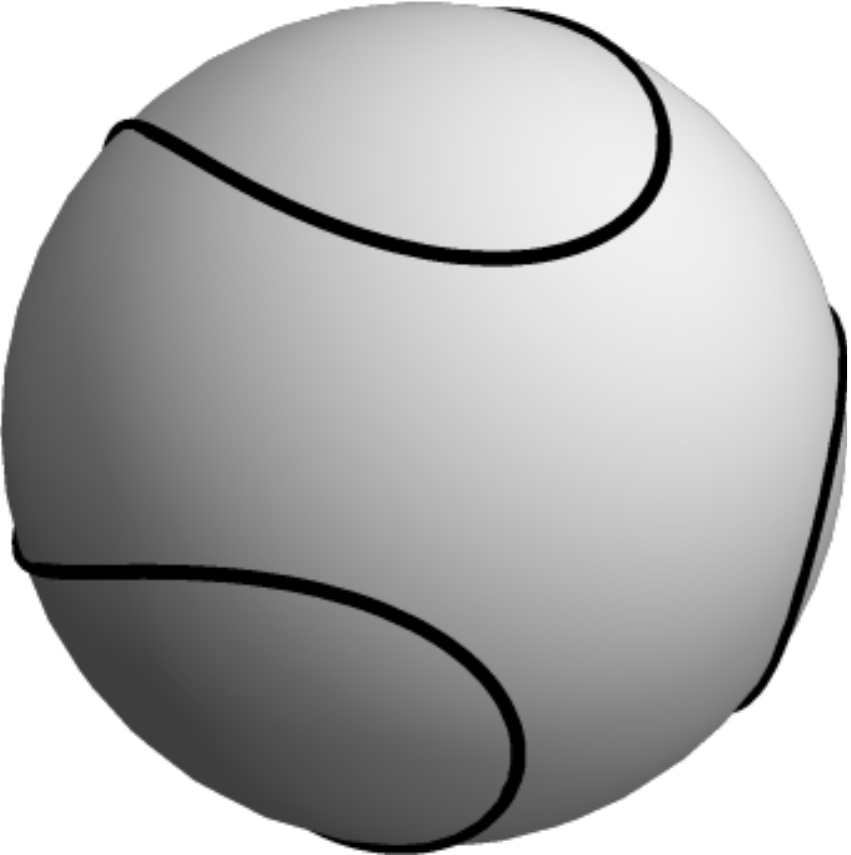}}
\caption{Curves in Example \ref{ex:3t}:  $\Gamma^{(2,a_2)}$ and
  $\gamma^{(3,a_3)}$  are smooth curves, and $\gamma ^{(2,a_2)}$ resembles the seam of a tennis ball.}\label{fig:approx 3 design}
\end{figure}

\section{Preparatory results}\label{sec:general preparation subspheres}
We will derive two preparatory results that are needed for our main
theorems in subsequent sections. The first one  is about the
connectivity of a graph associated to  a covering,  and  the second
one  is about the integration along the boundary of spherical caps.

\subsection{Connectivity of graphs associated to coverings}
The spherical cap of radius $0<r<\frac{\pi}{2}$ centered at $x\in\S^d$ is 
\begin{equation*}
	B_r(x)=\{z\in\S^d: \dist(x,z)\leq r\}.
\end{equation*}
To every  finite set $X\subseteq \S^d$   and  $r>0$  we 
associate a   graph $\mathcal{G}_{r}$ as follows:   its  vertices are the points of 
$X$. Two points  $x,y\in X$ are connected by an edge if   $B_{r}(x)\cap
B_{r}(y) \neq \emptyset$. 
\begin{lemma}\label{lemma:G connection 0}
Let $\rho$ be the covering radius of $X$. If $r\geq \rho$, then the graph $\mathcal{G}_r$ is connected.
\end{lemma}
\begin{proof}
Since $r\geq \rho$, the covering property 
$$
\bigcup_{x\in X} B_{r}(x) = \S^d
$$
holds. We fix  $x_0\in X$ and consider its connected component
$\mathcal{C}$, which consists of  all vertices connected to $x_0$ by some path. The
set $\S^d\setminus\bigcup_{x\in \mathcal{C}} B_{r}(x)$ is open since
$\bigcup_{x\in \mathcal{C}} B_{r}(x)$ is closed. If $\S^d\setminus
\bigcup_{x\in \mathcal{C}} B_{r}(x)\neq \emptyset$, then there is a
sequence $(y_n)_{n\in\N}\subset \S^d\setminus \bigcup_{x\in
  \mathcal{C}} B_{r}(x)$ such that $y_n\rightarrow y\in
\bigcup_{x\in \mathcal{C}} B_{r}(x)$. This means that $y\in
B_{r}(x)$ for some $x\in \mathcal{C}$.  Each $y_n$ is contained in a
spherical cap $B_{r}(x)$ for some $x\in X\setminus \mathcal{C}$. Since there are only
finitely many, there is a subsequence $y_{n_k}$ and some $\tilde{x}\in
X \setminus \mathcal{C}$ such that $y_{n_k} \in B_{r}(\tilde{x})$
and $y_{n_k} \to y$. 
Consequently, for given $\epsilon >0$ and $k$ large enough,
$$
\dist(\tilde{x},x) \leq \dist(\tilde{x},y_{n_k}) +
\dist(y_{n_k},y) + \dist(y,x) \leq r + \epsilon + r.
$$
This implies $\dist(\tilde{x},x)\leq 2r$ and, hence, $B_{r}(x) \cap
B_{r}(\tilde{x}) \neq \emptyset $. Since $x\in \mathcal{C}$, then by
definition of $\cG_r$ also $\tilde{x} \in \mathcal{C}$.  This, however, contradicts the earlier observation $\tilde{x}\in
X \setminus \mathcal{C}$. Hence, we
derive  
$
\bigcup_{x\in \mathcal{C}} B_{r}(x) = \S^d
$.

Let $y\in X$ be arbitrary. By the covering property, $y\in B_r(x)$ for
some $x\in \mathcal{C}$, and thus  $B_r(y) \cap B_r(x) \neq \emptyset
$. Consequently, $y\in \mathcal{C}$ and $\mathcal{C} = X$. 
This means that  the graph  $\mathcal{G}_{r}$ is connected, as
claimed. 
\end{proof}
\subsection{Integration along the boundary of spherical caps}
 Due to \eqref{eq:dist def}, the boundary of a spherical cap $B_{r}(x)$ is 
\begin{align}
	\partial B_{r}(x)& = \{z\in\S^d : \dist(x,z) = r\} \nonumber\\
	 & = \{z\in\S^d : \langle x,z\rangle = \cos r\}.\label{eq:B r}
\end{align}
By the  Pythagorian Theorem, $	\partial B_{r}(x)$ is a $d-1$-dimensional Euclidean sphere of radius $\sin r$
centered at $x\cos r$ in a hyperplane perpendicular to $x$, i.e., 
\begin{equation}\label{eq:Sdr}
\partial B_{r}(x) = \{z\in\S^d : \|x\cos r-z\| = \sin r\}.
\end{equation}
It is the intersection of the sphere $\S^d$ with a suitable hyperplane and we enforce the normalization
\begin{equation}\label{eq:cd0}
\int_{\partial B_{r}(x)}1 = 1.
\end{equation}

Next, we specify distinct functions that we integrate along $\partial B_{r}(x)$. Let $\mathcal{H}^d_k$ be the vector space of spherical harmonics of
degree $k\in\N$, i.e., the eigenspace of the Laplace-Beltrami operator on $\S^d$ with respect to the eigenvalue $-k(k+d-1)$, $k\in\N$, see, e.g., \cite{Stein:1971kx} for background material. The dimension of $\mathcal{H}^d_k$ is 
\begin{equation}\label{eq:dim H}
	\dim(\mathcal{H}^d_k)=\frac{2k+d-1}{d-1}\binom{k+d-2}{d-2},
\end{equation}
and the orthogonality relations
\begin{equation}\label{eq:orth H}
\mathcal{H}^d_k\perp \mathcal{H}^d_l,\qquad k,l\in\N,\quad k\neq l,
\end{equation}
 hold. The space $\bigoplus_{ k\leq t}\mathcal{H}^d_k$ coincides with the restriction of $\Pi_t$ to the sphere $\S^d$. When integrating over $\S^d$ or subsets or along curves in $\S^d$, we may therefore replace $\Pi_t$ by $\bigoplus_{ k\leq t}\mathcal{H}^d_k$. The proofs of our main results in Sections \ref{sec:sphere} and
\ref{sec:sphere II} rely on the following key observation. 
\begin{lemma}\label{lemma:Samko analog}
There are numbers $c_{d,k}(r)>0$ such that 
\begin{equation}\label{eq:from Samko}
\int_{\partial B_{r}(x)} f = c_{d,k}(r)f(x), \qquad \text{ for all }
f\in\mathcal{H}^d_k \,\, \text{ and } \,  x\in\S^d.
\end{equation}
\end{lemma}
In particular, the  normalization \eqref{eq:cd0} leads to
$c_{d,0}(r)=1$. The identity \eqref{eq:from Samko} is stated by Samko in
\cite[(1.37)]{Samko:1983oa} and referred to as a variant of the Cavalieri
principle \cite[Remark 4]{Samko:1983oa}. The analogue for the complex
sphere is mentioned in \cite{Oliveira:2015aa}. 

Here we provide a new proof of Samko's formula that is based on an
inductive construction of an orthonormal basis for $\cH ^d_k$. For
these facts we 
follow ~\cite{Muller:1966aa}. 
\begin{proof}
For $x\in\S^d$, we write 
\begin{equation*}
x=x_{d+1}e_{d+1}+\sqrt{1-x_{d+1}^2} \bar{x},\qquad \bar{x}\in\S^{d-1}.
\end{equation*}
Let $\{X^{d-1,m}_l\}_{m=1}^{\dim(\mathcal{H}^{d-1}_l)}$ be an orthonormal basis for $\mathcal{H}^{d-1}_l$. We denote the associated Legendre functions by $P^{d,l}_k$, $l=0,\ldots,k$, cf.~\cite{Muller:1966aa}. Then 
\begin{equation*}
Y^{d,l,m}_k(x) = P^{d,l}_k(x_{d+1}) X^{d-1,m}_l(\bar{x}),\qquad l=0,\ldots,k,\quad m=1,\ldots,\dim(\mathcal{H}^{d-1}_l),
\end{equation*}
form an orthogonal basis~\footnote{Note that for the next step of the
  induction we need to relabel the $Y^{d,l,m}_k$ as $X^{d,m}_l$.} for $\mathcal{H}^d_k$ and 
\begin{equation}\label{eq:Y in P and delta}
Y^{d,l,m}_k(e_{d+1}) = a_{d,k} \delta_{l,0}, 
\end{equation}
with suitable normalization constants $a_{d,k}$, cf.~\cite[Lemma
15]{Muller:1966aa}.

We first verify that $\int _{ \partial B_r(x) } f = c_{d,k}(r) f(x) $
for $x=e_{d+1}$. In this case $ \partial B_r(e_{d+1}) = \{ (z' \sin r,
\cos r): z' \in \sdd \}$, and the homogeneity of $X_l^{d-1,m}$ yields
\begin{align*}
\int_{\partial B_r(e_{d+1})} Y^{d,l,m}_k &= P^{d,l}_k(\cos r)
   \int_{\S^{d-1}}X^{d-1,m}_l(\sin r \, \cdot )
  \\
  &=  P^{d,l}_k(\cos r) \sin ^l r\int_{\S^{d-1}}X^{d-1,m}_l \\
& = P^{d,0}_k(\cos r) \sin ^l r\,\, \delta_{l,0}.
\end{align*}
After comparing with  \eqref{eq:Y in P and delta}, we set
$c_{d,k}(r):=P^{d,0}_k(\cos r) \sin ^l r /  a_{d,k}$ and obtain 
\begin{equation*}
\int_{\partial B_r(e_{d+1})} Y^{d,l,m}_k = c_{d,k}(r) Y^{d,l,m}_k(e_{d+1}).
\end{equation*}
For $k=0$, $Y_0^{d,0,0}$ is constant, and the normalization $\int _{
  \partial B_r(e_{d+1})} 1 = 1$ implies that $c_{d,0}(r) = 1$ for all
$r$. 
Thus  we have verified that \eqref{eq:from Samko} holds for all $f\in\mathcal{H}^d_l$ and $x=e_{d+1}$. 

We now consider general $x\in\S^d$. There is a rotation matrix $O$
such that $x=Oe_{d+1}$ and $\partial B_r(x) = \partial B_r(O e_{d+1})
= O \partial B_r(e_{d+1})$. This leads to
\begin{align*}
\int_{\partial B_r(x)} f  =  \int_{O \partial B_r(e_{d+1})} f  = \int_{\partial B_r(e_{d+1})} f\circ O.
\end{align*}
Since $\mathcal{H}^d_l$ is orthogonally invariant, $f\circ O\in \mathcal{H}^d_l$ and our observations for $e_{d+1}$ imply 
\begin{equation*}
\int_{\partial B_r(x)} f  = c_{d,k}(r) (f\circ O) (e_{d+1})=  c_{d,k}(r) f(x).\qedhere
\end{equation*}
\end{proof}
On $\S ^2$ Samko's formula connects point evaluations to line
integrals and thus gives a first hint of how quadrature formulas might
be related to $t$-design curves.

\section{Asymptotically optimal $t$-design curves in $\S ^2$}\label{sec:sphere}
We now state our first main result, which is Theorem~\ref{tint2} of
the Introduction. 
\begin{thm}\label{thm:t-design}
	There is a sequence of asymptotically optimal $t$-design
        curves in $\S^2$, i.e., there exists a sequence of piecewise
        smooth, closed curves $\gamma _t:[0,1]\rightarrow \S ^2$, for $t\in \N $,
of length $\ell (\gamma _t) \asymp t $,         such that
        $$\frac{1}{\ell(\gamma _t)}\int _{\gamma _t} f = \int _{\S ^2} f\qquad \text{ for all
        } f\in \Pi _t \, .$$
The trajectory of every $\gamma _t$ is a union of Euclidean circles.          
\end{thm}
\begin{proof}
According to \eqref{eq:Sdr} with $d=2$, the boundary of the spherical cap $B_r(x)=\{z\in\S^2 : \dist(x,z)\leq r\}$ is a Euclidean circle of radius $\sin r$ given by 
\begin{equation}\label{eq:Ga in proof}
	\Gamma_{x,r} = \{z\in\S^2 : \|x\cos r-z\|=\sin r\}\subset\S^2,
\end{equation}
cf.~Figure \ref{fig:sphere00}. The asymptotically optimal $t$-design
curve will be constructed from a union  of circles $\Gamma_{x,r}$,
where $x$ runs through a set of $t$-design points and $r$ is
essentially their covering radius. For each circle we choose an
(arbitrary) orientation and obtain a closed curve $\gamma _{x,r}$ with
trajectory $\Gamma _{x,r}$. The formal sum $\gamma _t^r = \dotplus
_{x\in X_t} \gamma _{x,r}$ is usually called a cycle with trajectory
$\bigcup _{x\in X_t} \Gamma 
_{x,r}$ and integral $\int _{\gamma _t^r} f = \sum _{x\in X_t} \int
_{\gamma _{x,r}} f$ ~\cite{rudin87}.

Note that in $\S ^2$ the
submanifold $\partial B_r(x)$ and the circle $\Gamma _{x,r}$ coincide, so that the normalization~\eqref{eq:cd0} leads to $\int_{\Gamma_{x,r}} f = \frac{1}{\ell(\gamma_{x,r})}\int_{\gamma_{x,r}}f$, where $\gamma_{x,r}$ is the actual curve that traverses $\Gamma_{x,r}$. 
In higher dimensions this slight inconsistency between $\Gamma _{x,r}$
and $\partial B_r(x)$ no longer
occurs. 

(i) According to \cite{Bondarenko:2011kx}, there is a sequence
$(X_t)_{t\in\N}$ of asymptotically optimal $t$-design points in
$\S^2$, i.e., there exist  finite sets $X_t\subset\S^2$, such that
$|X_t|\asymp t^2$  and 
\begin{equation}\label{eq:design ingredient}
\frac{1}{|X_t|}\sum_{x\in X_t} f(x)=\int_{\S^2} f,\qquad \text{ for
  all }  f\in
\bigoplus_{k\leq t}\mathcal{H}^2_k\, .
\end{equation}

(ii) Let  $r>0$, $\Gamma^r_t:=\bigcup_{x\in X_t}\Gamma_{x,r}$,
and $\gamma^r_t:=\dotplus _{x\in X_t}\gamma_{x,r}$ be the corresponding
cycle. We first show that this cycle provides exact integration on
$\Pi _t$. 

For the component $\mathcal{H}^2_0=\spann\{1\}$, we observe 
\begin{equation*}
\frac{1}{\ell(\gamma^r_t)} 	\int_{\gamma^r_t}1= 1=\int_{\S^d} 1.
\end{equation*}
We now consider $f\in \mathcal{H}^2_k$ for $1\leq k\leq t$. Lemma \ref{lemma:Samko analog} with $c_k(r):=c_{2,k}(r)$ and the definition of $t$-design points yield
\begin{align*}
\frac{1}{\ell(\gamma^r_t)} 	\int_{\gamma^r_t}f & = \frac{1}{|X_t|}\sum_{x\in X_t}
       \frac{1}{\ell(\gamma_{x,r})} \int_{\gamma_{x,r}}f\\
       & = \frac{1}{|X_t|}\sum_{x\in X_t}
 \int_{\Gamma_{x,r}}f\\
       & = \frac{c_k(r)}{|X_t|}\sum_{x\in X_t} f(x)\\
       &=c_k(r) \int_{\S^2} f.
\end{align*}       
Since  $\mathcal{H}^2_0\perp\mathcal{H}^2_k$ for $k=1,2,\ldots$,
cf.~\eqref{eq:orth H},  we obtain  $\int_{\S^2} f=0$ for  $f\in
\mathcal{H}^2_k$, and  the factor $c_k(r)$ does not matter. We derive
\begin{equation}\label{eq:again}
\frac{1}{\ell(\gamma^r_t)} 	\int_{\gamma^r_t}f  = \int_{\S^2} f
\qquad \text{ for all } f\in \cH _k^2,    
\end{equation}
and, by linearity,  exact quadrature holds for all $f\in 
\bigoplus _{k=0}^t \cH _k^2 $.

(iii) Identity~\eqref{eq:again} holds for every $0<r<\frac{\pi}{2}$,
and thus every cycle $\gamma _{t}^r$ (formal sum of closed curves)
yields exact integration.    We now determine
radii $r=r_t$ depending on the degree $t$, so that $\Gamma_t:=\Gamma^{r_t}_t=\bigcup_{x\in
  X_t}\Gamma_{x,r_t}$ is indeed the trajectory of a single continuous
closed  curve. 

Let $\rho_t = \sup_{x\in\X}\left(\inf_{y\in X_t}\dist_{\X}(x,y)\right)$
be  the covering radius of $X_t$ as in  \eqref{eq:cov rad}. 
We choose $r_t:=\rho_t$, so that
 \begin{equation}\label{eq:S2 cov B}
 	\S^2=\bigcup_{x\in X_t} B_{r_t}(x).
 \end{equation}
The circles $\Gamma_{x,r_t}=\partial B_{r_t}(x)$, for $x\in X_t$,
induce a graph $\mathcal{G}$   as
follows: the vertices of $\cG $ 
 are the intersection points of the $\Gamma _{x,r_t}$,  and
its  edges are associated to arcs
on these circles between the intersection points, cf.~Figures
\ref{fig:sphere00} and \ref{fig:sphere}. For each circle
$\Gamma_{x,r_t}$ we have  fixed  an orientation. 
As a result of this construction  we obtain a directed graph
$\mathcal{G}_{\rightarrow}$~\cite{Wilson:1998qa}.

\begin{figure}
\centering
\includegraphics[width=.38\textwidth]{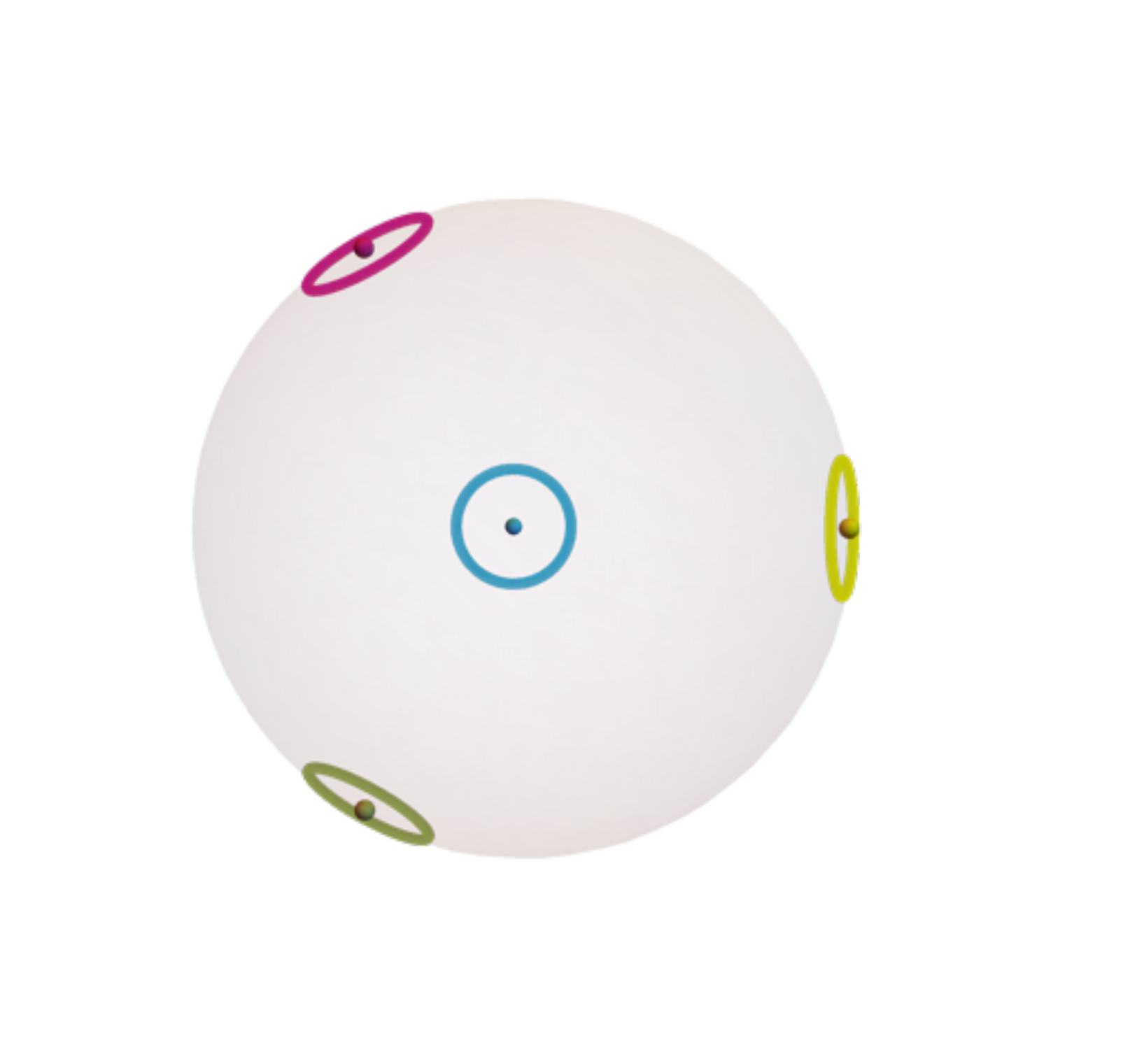}\hspace{-1.5cm}
\includegraphics[width=.38\textwidth]{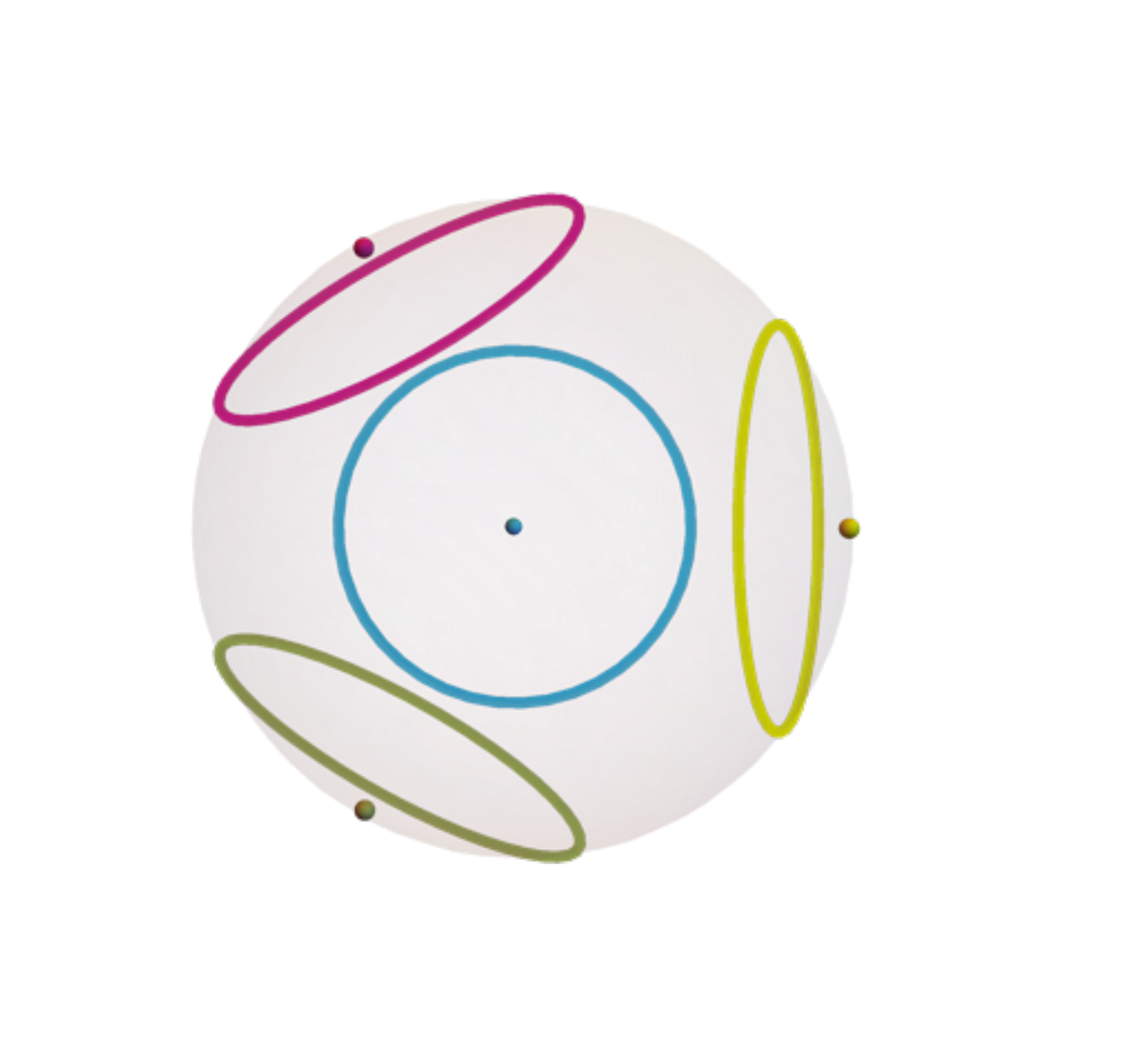}\hspace{-1.5cm}
\includegraphics[width=.38\textwidth]{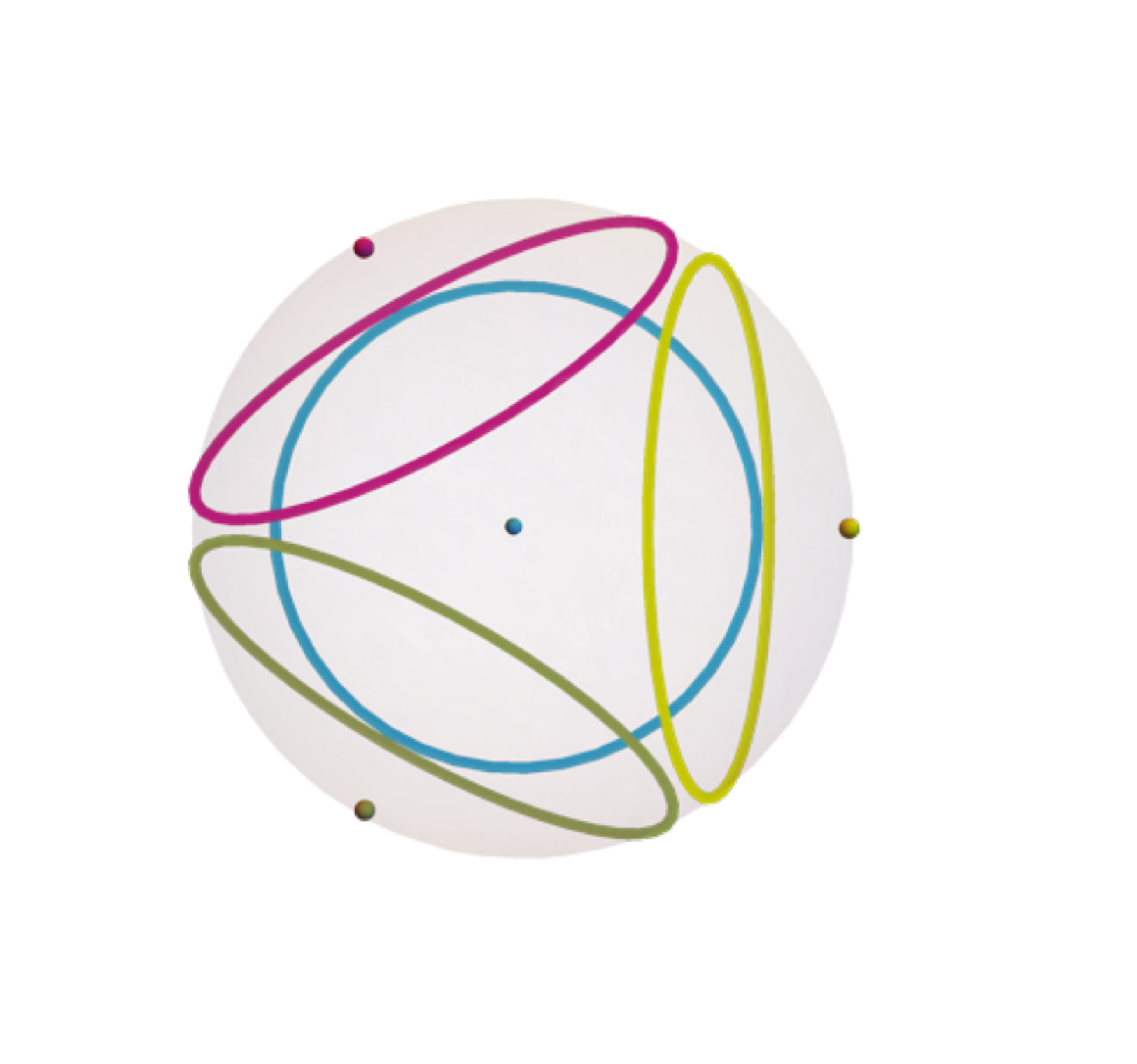}

\includegraphics[width=.38\textwidth]{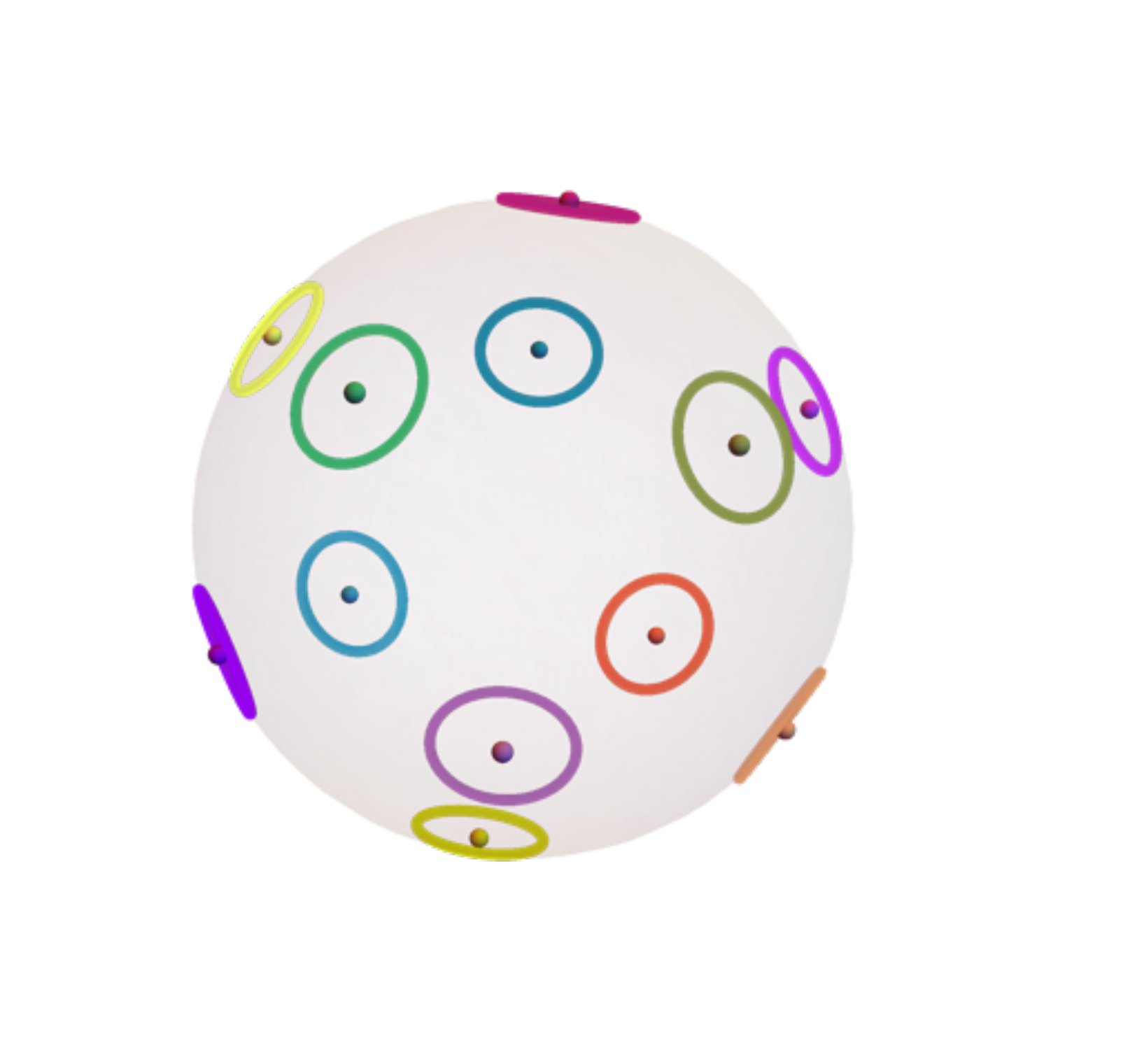}\hspace{-1.5cm}
\includegraphics[width=.38\textwidth]{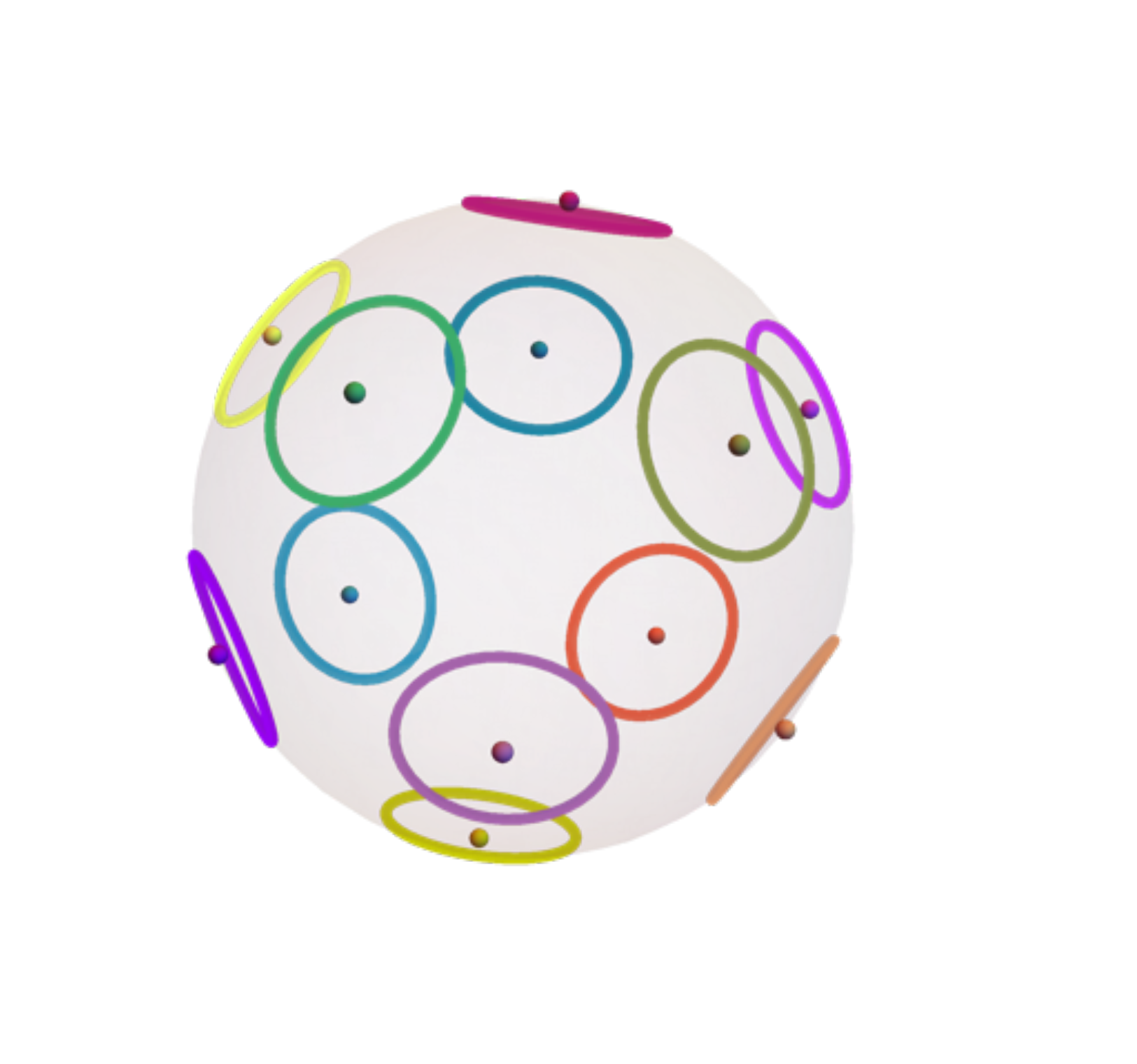}\hspace{-1.5cm}
\includegraphics[width=.38\textwidth]{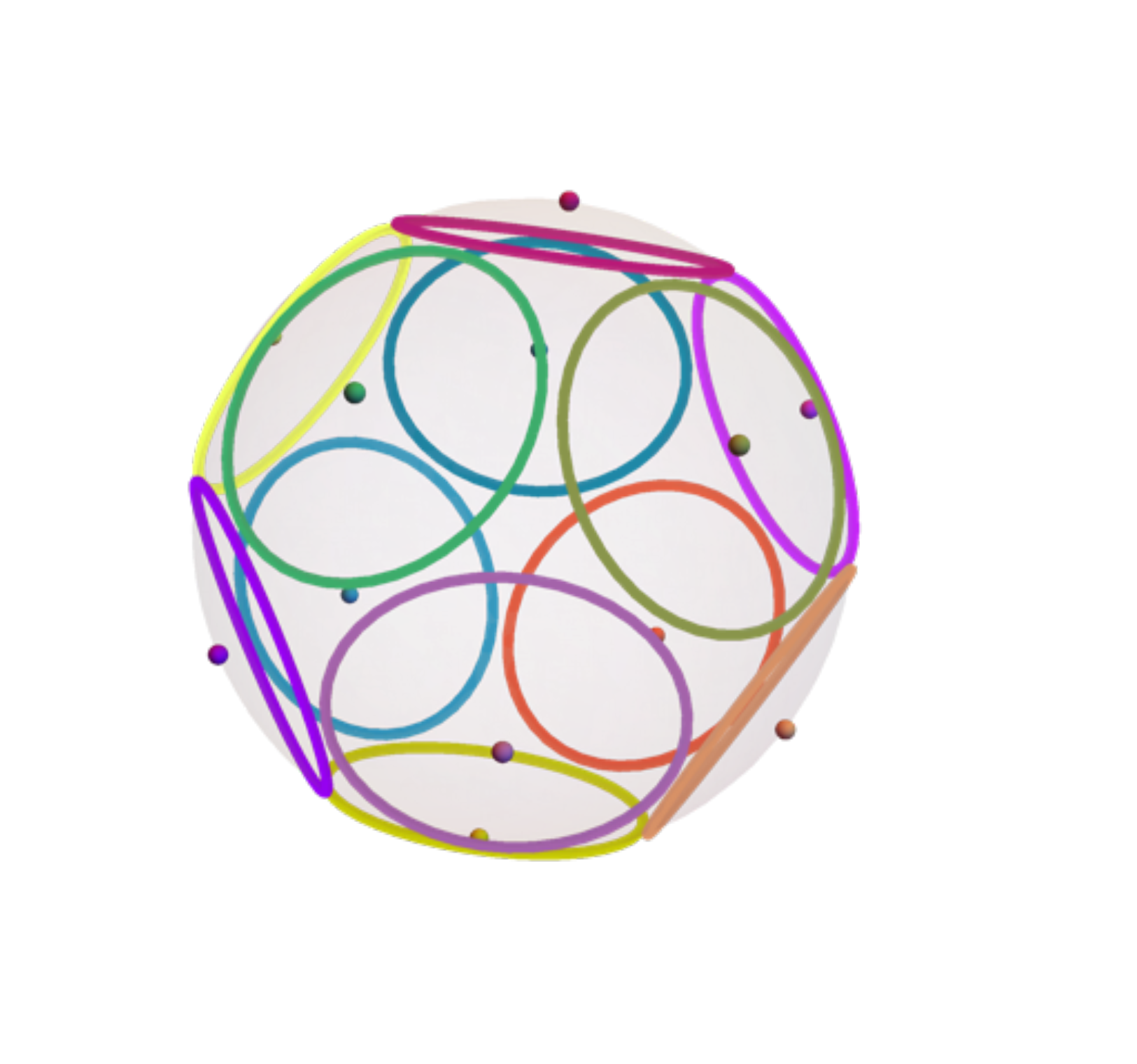}
\caption{Circles centered at $t$-design points with radii increasing
  from left to right. The  top row shows $2$-design points and the
  bottom row $5$-design points.
}\label{fig:sphere00}
\end{figure}
\begin{figure}
\centering
\includegraphics[width=.6\textwidth]{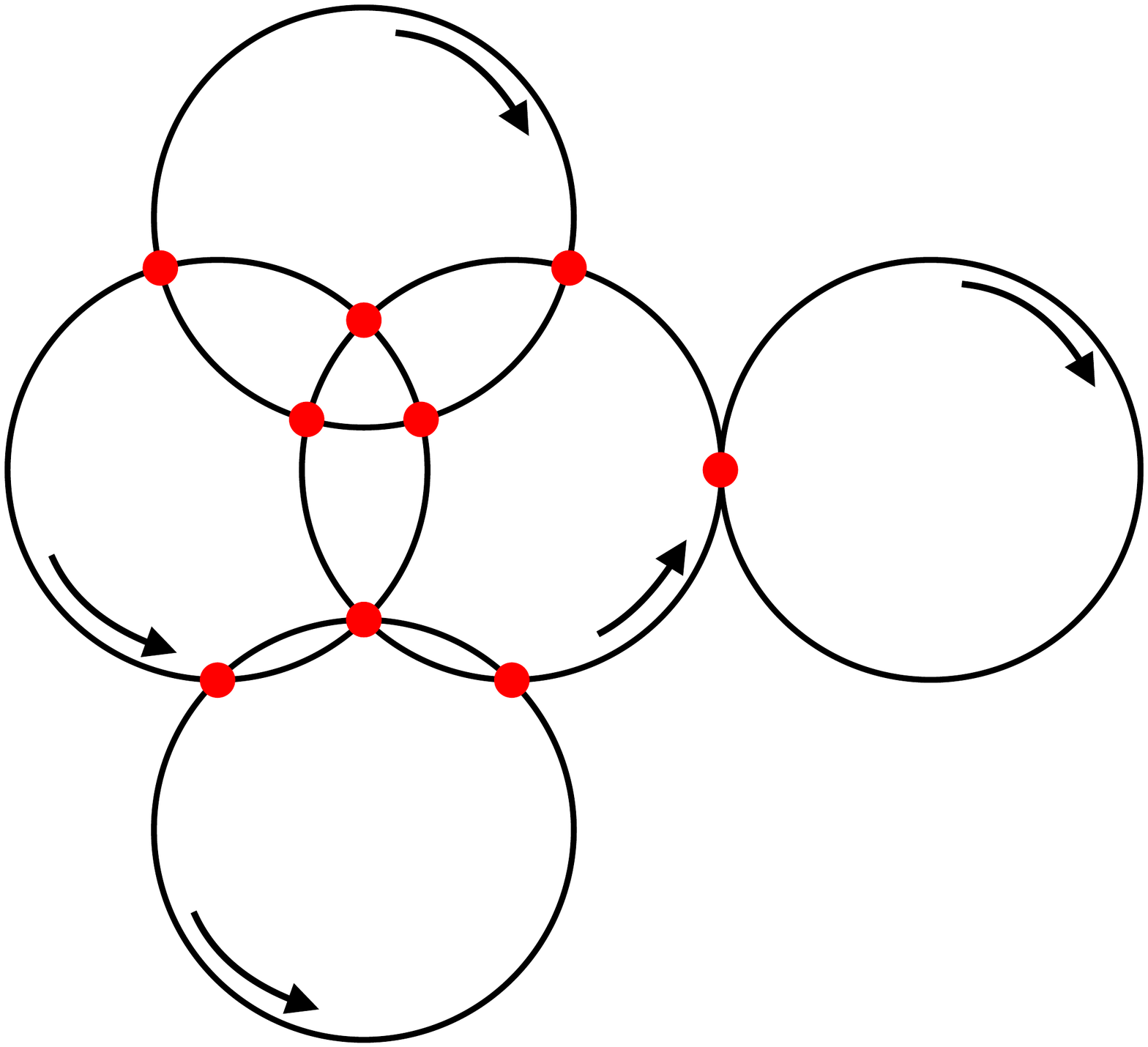}
\caption{A set of connected circles induces a directed graph, whose
  vertices are the intersection points and whose  edges are the
  associated arcs 
   of the circles. Each vertex has equal in-degree and
  out-degree, namely the number of circles running through this
  point. 
}\label{fig:sphere}
\end{figure}

\begin{lemma}\label{lemma:G connection}
The graph $\mathcal{G}$ is strongly connected, i.e., the directed graph $\mathcal{G}_{\rightarrow}$ is connected. 
\end{lemma}
\begin{proof}
We first verify that the undirected graph $\mathcal{G}$ is
connected. Pick two arbitrary but distinct vertices
$v_1,v_2\in\mathcal{G}$. Then there are points $x_1,x_2\in X_t$ such
that $v_i\in \partial B_{r_t}(x_i)$ for $i=1,2$. If $x_1=x_2$, then
$v_1,v_2$ are connected in $\mathcal{G}$ since they lie on the same
circle.

We may thus assume that  $x_1\neq x_2$.   In this case we consider the
auxiliary  graph $\mathcal{G}_{r_t}$ treated in Lemma~\ref{lemma:G
  connection 0}. Its  vertices are the points of the $t$-design  
$X_t$. Two points  $x,y\in X_t$ are connected by an edge  if and only if $B_{r_t}(x)\cap
B_{r_t}(y) \neq \emptyset$. According to  Lemma \ref{lemma:G connection 0} and $r_t\geq \rho_t$, the graph $\mathcal{G}_{r_t}$ is connected. Thus, there is a path from $x_1$ to $x_2$ in $\mathcal{G}_{r_t}$, say $x_1=x_{i_1},\ldots,x_{i_m}=x_2$, so that $B_{r_t}(x_{i_j})\cap B_{r_t}(x_{i_{j+1}})\neq\emptyset$. This also implies $\partial B_{r_t}(x_{i_j})\cap \partial B_{r_t}(x_{i_{j+1}})\neq\emptyset$. Clearly, vertices $v_{i_j},v_{i_{j+1}}$ of $\mathcal{G}$ with $v_{i_j}\in \partial B_{r_t}(x_{i_j})$ and $v_{i_{j+1}}\in \partial B_{r_t}(x_{i_{j+1}})$ are connected in $\mathcal{G}$. Eventually, there is a path from $v_1$ to $v_2$ in $\mathcal{G}$. 

We still need to verify that the directed graph
$\mathcal{G}_{\rightarrow}$ is also connected. If two distinct
vertices $v_{i_j},v_{i_{j+1}}$ lie on the same circle $\partial
B_{r_t}(x_1)$, then we simply get from $v_{i_j}$ to $v_{i_{j+1}}$ by
following the chosen orientation of $\partial B_{r_t}(x_1)$. 
No further difficulties arise and we deduce that $\mathcal{G}_{\rightarrow}$ is connected.
\end{proof}
By construction, each vertex of $\mathcal{G}_{\rightarrow}$ has as many 
incoming as outgoing edges, cf.~Figure \ref{fig:sphere}. Euler's Theorem about directed, strongly connected graphs implies that
there is an Euler cycle~\cite{Wilson:1998qa}. Hence, all circles
in $\bigcup _{x\in X_t} \partial B_{r_t}(x)=\bigcup_{x\in
  X_t}\Gamma_{x,r_t}$ can be traversed by a
single, closed, piecewise smooth  curve $\gamma_t$ 
on $\mathbb{S}^2$  with 
trajectory $\Gamma_t=\bigcup_{x\in
  X_t}\Gamma_{x,r_t}$. 

(iv) To compute the length of $\gamma _t$, we relate the covering
radius $r_t $ of $X_t$ to the degree $t$.  The same  proof as in
\eqref{eq:rho t etc} shows that $r_t \lesssim t^{-1}$,  see also \cite{Brandolini:2014oz,Breger:2016rc}. Therefore, 
$\ell(\gamma_{x,r_t})=2\pi \sin r_t\asymp r_t \lesssim t^{-1}$. This  leads to
\begin{equation*}
	\ell(\gamma_t)=|X_t| \ell(\gamma_{x,r_t})\asymp t^2 \, r_t\lesssim t.
\end{equation*}
In view of the lower bound for the length of $t$-design curves in
Theorem~\ref{thm:t-design length}, 
our construction yields a  sequence of
asymptotically optimal $t$-design curves. 
\end{proof}

\section{General existence of spherical $t$-design curves}\label{sec:sphere II}

We now prove  the existence  of $t$-design curves in $\S^d$ for
all $t\in\N$ in dimension  $d\geq 2$. This is Theorem~\ref{tint3} of the Introduction.

\begin{thm}\label{thm:Sd existence}
	Let $d\geq 2$ be an arbitrary integer. There is a sequence $(\gamma_t)_{t\in\N}$ of $t$-design curves in $\S^d$ such that
	\begin{equation*}
		\ell(\gamma_t) \lesssim t^{\frac{d(d-1)}{2}} .
              \end{equation*}
        Furthermore, the trajectory of every $\gamma _t$ consists of a union of
        Euclidean circles.       
      \end{thm}
      We will  prove this theorem by induction on the dimension
      $d$. Similar to the proof of Theorem~\ref{thm:t-design}  we
      first show the existence of a  cycle (a formal sum of closed
      curves) that yields exact 
      integration. Then we use a combinatorial argument to build a 
connected  trajectory. 

\subsection{Some geometry on the sphere}
Let us first state few simple observations about  the action of the orthogonal group on $\S^d$. 
\begin{lemma}\label{lemma:easy 0}
Let  $d\geq 2$  and $A,B\subseteq\R^d$  be two  finite sets. If $0\not\in A$, then there is a rotation $O\in\mathcal{O}(d)$ such that $OA\cap B=\emptyset$.
\end{lemma}
\begin{proof}
For the sake of completeness  we provide the simple arguments. 

The claim is verified by induction. For the induction step we identify
$\cO (d)$ as a subgroup of $\cO (d+1)$ via the homomorphism
 $ \bar{O} \in \cO (d)  \mapsto j(\bar{O})\in \cO (d+1)$,
 $j(\bar{O})(x) = (\bar{O}\bar{x}, x_{d+1})  $ for $x =
 (\bar{x},x_{d+1})\in \R ^{d+1}$.

The case $d=2$ is obvious.  Consider now $A,B\subseteq\R^{d+1}$. Let $\bar{A},\bar{B}\subseteq \R^d$ be the orthogonal projections of $A,B$ onto the first $d$ coordinates. 

Case $1$:  $0\not \in \bar{A}$. By  the induction hypothesis  there
exists $\bar{O}\in\mathcal{O}(d)$ such that $\bar{O}\bar{A}\cap
\bar{B}=\emptyset$. Then $O= j(\bar{O})\in \mathcal{O}(d+1)$ satisfies
$OA\cap B = \emptyset$.

Case $2$:  $0\in \bar{A}$. Since $0\not\in A$, there is
$O_1\in\mathcal{O}(d+1)$ such that $O_1A \cap (\R
e_{d+1})=\emptyset$. Let $A_1 = O_1A$ and $\bar{A_1}$ be the
projection onto the first $d$ coordinates. 
Then   $0\not\in \overline{O_1 A}$, and by  the induction hypothesis
there is $\bar{O}_2\in\mathcal{O}(d)$ such that
$\bar{O}_2\overline{A_1 } \cap \bar{B}=\emptyset$. Then $O =
j(\bar{O}_2) O_1$ satisfies $OA \cap B = \emptyset $. 
\end{proof}

\begin{corollary}\label{lemma:easy 1}
Let $d\geq 2$, $v\in \S ^d$,  and  $A,B\subseteq\S^{d}$ two  finite
sets. If $\pm v\not\in A$, then there is a rotation
$O\in\mathcal{O}(d+1)$ such that $Ov=v$ and $OA\cap B=\emptyset$. 
\end{corollary}
\begin{proof}
Without loss of generality we may assume that $v=e_{d+1}\in \S ^d$. 
Let $\bar{A},\bar{B}\subseteq \R^d$ be the orthogonal projections of
$A,B$ onto the first $d$ coordinates. Since $\pm e_{d+1}\not\in A$, we
deduce $0\not\in \bar{A}$. According to Lemma \ref{lemma:easy 0},
there is $\bar{O}\in\mathcal{O}(d)$ such that $\bar{O}\bar{A}\cap
\bar{B}=\emptyset$. The  $O= j(\bar{O}) \in\mathcal{O}(d+1)$ does the
job. 
\end{proof}

As in Section \ref{sec:general preparation subspheres}, we now consider spherical caps and their boundary
\begin{align*}
	B_{r}(x)&=\{y\in\S^d : \dist_{\S^d}(x,y)\leq r\},\\
	\partial B_{r}(x) &=  \{z\in\S^d : \|x\cos \rho_t-z\| = \sin r\},
\end{align*}
see \eqref{eq:Ga in proof} for $d=2$. Clearly $\partial B_{r}(x)$ is
homeomorphic (diffeomorphic)  to the  $d-1$-dimensional sphere $\S
^{d-1}$. For our analysis we will use the following homeomorphism.

Let $x\in\S^d$, $0<r<\frac{\pi}{2}$, and $O = O_x \in \mathcal{O} (d+1)$ a 
matrix in the orthogonal group  acting  on $\R^{d+1}$ such that $x=Oe_{d+1}$, where $e_{d+1}=(0,\ldots,0,1)^\top\in\mathbb{R}^{d+1}$. Now set 
\begin{equation}
  \label{eq:c10}
  \phi_{x,r}(z) =O\begin{pmatrix}
\sin r z\\
\cos r
\end{pmatrix} \qquad z\in \S ^{d-1} \, 
\end{equation}
and recall that $e_d\in\mathbb{R}^{d}$ is the north pole in $\S^{d-1}$.
\begin{lemma}\label{lemma:induction curve sphere}
The map $\phi _{x,r}: \S ^{d-1} \to \S ^d$ has the following
properties.

(i) $\phi _{x,r}$ is a diffeomorphism between $\S^{d-1}$ and $\partial
B_r(x)$.

(ii) Let $x,y\in \S ^{d-1}$ with $\dist_{\S^d}(x,y) < r\leq \frac{\pi}{2}$. Then there
is a unique radius 
$\sigma \in [\pi /3, \pi /2] $, such that 
$$
\phi _{x,r} \big( \partial B_\sigma (e_d) \big) = \partial B_r(x) \cap
\partial B_r(y) \, .
$$
In particular, the intersection $\partial B_r(x) \cap
\partial B_r(y) $ is diffeomorphic to $\S ^{d-2}$. 
\end{lemma}
\begin{proof}[Proof of Lemma \ref{lemma:induction curve sphere}]
(i) If $\|z\|=1$, then $\| \phi _{x,r} (z) \| ^2 = \|z\|^2
\sin ^2 r + \cos ^2 r = 1$, and
\begin{align*}
\dist(x,\phi _{x,r}(z)) &= \dist(Oe_{d+1}, O(z \sin r,
\cos r ) \\
&= \arccos \langle e_{d+1}, (z \sin r,
\cos r )\rangle = r \, .
\end{align*}
Clearly, $\phi _{x,r}$ is a bijection between $\S ^{d-1}$ and
$\partial B_r(x)$.

(ii) For $z\in \R ^{d+1} $ we write $z=(z', z_{d+1}) = (z'',
z_d,z_{d+1})$ with $z' \in \R ^d $ and $z'' \in  \R ^{d-1} $. 
By  applying a suitable rotation, we may  assume without loss of
  generality that $x= e_{d+1}$ and $y = \sqrt{1-y_{d+1}^2} \, e_d +
  y_{d+1}e_{d+1}= (0,\sqrt{1-y_{d+1}^2},y_{d+1})$. Then the assumption $\dist(x,y) = \arccos \langle x,y\rangle
  <r$ yields $\langle x,y\rangle = y_{d+1} >\cos r$.

  Now take  $z\in \partial B_{r } (e_{d+1}) \cap
\partial B_{r } (y)$. By Step (i),  
$$
\partial B_r(e_{d+1}) = \phi
_{e_{d+1},r}(\sdd) = \{ (z' \sin r, \cos r: z' \in \sdd \},
$$
so
$z_{d+1} = \cos r $. Since $\dist (z,y) = \arccos \langle z,y
\rangle = r$ as well, we obtain

  \begin{align*}
\langle z,y \rangle =  
 z_d \sqrt{1-y_{d+1}^2} + z_{d+1}  y_{d+1} =  z_d \sqrt{1-y_{d+1}^2} +
    \cos r \,y_{d+1} = \cos r\, .
  \end{align*}
Solving for $z_d$,   this implies that 
  \begin{equation}
    \label{eq:chmas}
   z_d = \cos r \,
  \frac{1-y_{d+1}}{\sqrt{1-y_{d+1}^2}} \, .  
  \end{equation}
Let us abbreviate the occurring fraction by $\tau =
\frac{1-y_{d+1}}{\sqrt{1-y_{d+1}^2}}$. Then 
  \begin{equation}\label{eq:christmas}
  \begin{split}
  \|z''\|^2 &= 1-z_{d+1}^2 - z_d^2\\
  & = 1 - \cos  ^2 r - \tau ^2 \cos ^2 r
   \\
  &= \sin ^2 r - \tau ^2 \cos ^2 r   =:s ^2 \,. 
  \end{split}
  \end{equation}
We may switch from $z''$ to $sz''$ with $z''\in\S^{d-2}$, so that every point in $z\in \partial B_r(x) \cap \partial B_r(y)$ has
coordinates 
\begin{equation}
  \label{eq:c15}
  z= \begin{pmatrix}  s z'' 
    \\
 \tau \cos r \\
\cos  r
\end{pmatrix} \, ,\qquad z''\in \S ^{d-2}\,.
\end{equation}

By comparison,  a point $z'\in \partial B_\sigma (e_d)\subseteq \sdd $ is of the
  form $(z'' \sin \sigma , \cos \sigma )$ for $z'' \in \S
  ^{d-2}$. Consequently,
  \begin{equation}
    \label{eq:c14}
    \phi _{e_{d+1},r}(z') = \begin{pmatrix}  z'' \sin \sigma \,  \sin r
    \\
\cos \sigma \,\sin r\\
\cos r
\end{pmatrix} \, .
    \end{equation}
  
We have to show that every point in $ \partial B_{r } (x) \cap
\partial B_{r } (y)$ can be represented in this way. For 
\eqref{eq:c15} and \eqref{eq:c14} to represent the same set, we need to verify that 
 the following identities
\begin{equation}
  \label{eq:cc12}
  s=\sin \sigma \, \sin r \qquad \text{ and } \qquad  \cos r \,\tau = \cos
  \sigma \, \sin r 
\end{equation}
can be satisfied with a suitable choice of $\sigma $.
Clearly $\sigma $ is determined by 
\begin{equation}
  \label{eq:c13}
  \sin \sigma  = \frac{s}{\sin r }  \, .
\end{equation}
Then using \eqref{eq:christmas}
$$
\cos ^2\sigma \, \sin ^2 r  = (1-\sin ^2\sigma  ) \sin ^2 r = \sin
^2 r - s^2 = \tau ^2 \, \cos ^2r  \, ,
$$
and the second identity in \eqref{eq:cc12} is also satisfied.

Finally, since $y_{d+1} = \cos r $ and $\tau ^2 =
\frac{1-y_{d+1}}{1+y_{d+1}}$, we estimate the size of $\sigma $ as
\begin{align*}
  \sin ^2\sigma  = \frac{s^2}{\sin ^2 r} &= 1-\frac{\cos ^2 r}{\sin ^2r} \,
                     \frac{1-y_{d+1}}{1+y_{d+1}} \\
&\geq 1-\frac{\cos ^2 r}{\sin ^2r} \,
                     \frac{1-\cos r}{1+\cos r} \,.
\end{align*}
For $u=\cos r\geq 0$, we observe 
$$
1-\frac{\cos ^2 r}{\sin ^2 r} \,
                     \frac{1-\cos r}{1+\cos r} = 1-\frac{u^2}{1-u^2}\frac{1-u}{1+u}=1-\frac{u^2}{(1+u)^2}\geq \frac{3}{4}\,,
$$
so  that we obtain $\sin ^2\sigma \geq \frac{3}{4}$. Consequently $\pi/3 \leq \sigma \leq \pi
/2$, which means that $ B_\sigma (e_d)$ covers a fixed portion of
$\sdd $. 
\end{proof}

Next we check how curves in $\S ^{d-1}$, in particular, $t$-design
curves, are mapped by $\phi _{x,r}$.

\begin{lemma} \label{crux2}
    Suppose that $\gamma $ is  a $t$-design curve $\gamma$ in
    $\S^{d-1}$. Let $x\in\S^d$ and  
    $0<r<\frac{\pi}{2}$.
\begin{itemize}
\item[(i)] Then $\gamma_{x,r}=\phi _{x,r}
    \circ \gamma $ is a curve in $\partial B_r(x) \subseteq \S^d$ with
    length $\ell(\gamma_{x,r})= \ell(\gamma) \,  \sin r $, such
    that 
\begin{equation*}
\frac{1}{\ell(\gamma_{x,r})}\int_{\gamma_{x,r}}f = \int_{\partial
  B_{r}(x)} f,\qquad \text{ for all } f\in \Pi_t \, .
\end{equation*}
\item[(ii)] For a given point  $w\in \partial B_r(x)$, there is a $t$-design
curve $\gamma'$ in $\S ^{d-1}$, such that $w$ lies on $\phi
_{x,r} \circ \gamma '$. 
\item[(iii)] If $\gamma$ consists  a union of circles, then so does $\gamma_{x,r}$ and $\gamma'$ can also be chosen to do so.

\item[(iv)]  Assume that $\gamma$ and $\Psi\subseteq \S^d$ both consist of 
     a finite union of circles  and $w\in \Psi$. Then there
     exists a $t$-design curve $\gamma ''$ in $\sdd $ consisting of a
     union of circles,  such that $w$ is contained in the trajectory $\Gamma''_{x,r}$ of the curve $\phi
     _{x,r} \circ \gamma ''$ and the intersection $\Gamma ''_{x,r}
     \cap \Psi$ is      finite.
\end{itemize}
\end{lemma}

In the following, we refer to $\gamma_{x,r}$ as a $t$-design curve for $\partial B_{r}(x)$.

\begin{proof}
(i)   First we consider the north pole $x=e_{d+1}$. The curve $\gamma_{e_{d+1},r}=
\phi _{e_{d+1},r} \circ \gamma =  \left(\sin r \,\gamma , \cos r \right)^\top$
has arc length  $\ell(\gamma_{e_{d+1},r})=\sin r \,\ell(\gamma)$,  and we
derive 
\begin{align*}
\frac{1}{\ell(\gamma_{e_{d+1},r})}\int_{\gamma_{e_{d+1},r}}f & = \frac{1}{\sin r\, \ell(\gamma)}\int_{\sin r\,\gamma}f(\;\cdot\;,\cos r)\\
& = \frac{1}{\sin r\, \ell(\gamma)} \int_0^1 f(\sin r\,\gamma(s),\cos r)\|\sin r\, \dot{\gamma}(s)\| \mathrm{d}s\\
& = \frac{1}{\ell(\gamma)}\int_{\gamma}f(\sin r \, \;\cdot\;,\cos r).
\end{align*}
For $f\in \Pi_t$ and $y\in\S ^{d-1}$, the mapping $y  \mapsto f\circ \phi _{e_{d+1},r}(y) =
f(y \sin r ,\cos r)$ is a polynomial of degree $t$ restricted to
$\S^{d-1}$. The $t$-design property leads to 
\begin{align*}
\frac{1}{\ell(\gamma_{e_{d+1},r})}\int_{\gamma_{e_{d+1},r}}f & = \int_{\S^{d-1}}f(\sin r \,\;\cdot\;,\cos r)\\
& =\int_{\partial B_{r}(e_{d+1})}f\;,
\end{align*}
where the latter equality is due to the normalization $\int_{\S^{d-1}}1 = 1 = \int_{\partial B_{r}(e_{d+1})}1$, cf.~\eqref{eq:cd0}. 

For general $x\in\S^d$, there is a rotation matrix $O$ such that
$x=Oe_{d+1}$ and rotational invariance of $\dist$ yields $
B_r(x)=OB_r(e_{d+1})$ and $\partial B_{r}(x) = O\partial
B_{r}(e_{d+1})$. The curve $\gamma_{x,r}=O\gamma_{e_{d+1},r}=\phi
_{x,r} \circ \gamma $
satisfies $\|\dot{\gamma}_{x,r}\| = \|\dot{\gamma}_{e_{d+1},r}\|$. For
$f\in\Pi_t$, we also have  $f\circ O\in\Pi_t$ and deduce 
\begin{align*}
\int_{\partial B_{r}(x)} f = \int_{O\partial B_{r}(e_{d+1})} f & = \int_{\partial B_{r}(e_{d+1})} f\circ O\\
& = \frac{1}{\ell(\gamma_{e_{d+1},r})}\int_{\gamma_{e_{d+1},r}}f\circ O = \frac{1}{\ell(\gamma_{x,r})}\int_{\gamma_{x,r}}f.
\end{align*}

(ii) If $\gamma$ is a $t$-design curve in $\sdd$, then for every
orthogonal matrix $U \in \cO (d)$ the rotated curve  $\gamma'=U\gamma
$ is also a $t$-design curve in $\sdd$. By suitably choosing $U$, we
can always achieve that a given point $v\in \sdd$ lies on the trajectory of $U\gamma
$. Now let $w \in \partial   B_r(x)\subseteq \S ^d$ and $v\in \sdd$ its pre-image under $\phi
_{x,r}$.  Consequently $w=\phi _{x,r}(v) $ lies on the curve $\gamma'_{x,r} = \phi _{x,r} \circ
(U\gamma)$.

(iii) Clearly, if $\gamma $ is a union of (Euclidean)  circles in $\sdd$, then
$\gamma_{x,r}=O(\sin r\,\gamma,\cos r)^\top$ is a union of
Euclidean circles in $\S ^d$. The same holds for $\gamma'=U\gamma
$.

(iv) Let $v =  \phi _{x,r} ^{-1}(w)$ and  let $\Psi_0 = \phi _{x,r} ^{-1} \circ (\Psi\cap \partial
  B_{r}(x)) $. Then $\Psi _0$  is again a finite union of
  circles or arcs of circles.  Two circles are either disjoint, or they intersect in one
  or two points, or they coincide. This can happen only when the two
  circles have the same center. Let $C_\gamma \subseteq \sdd $ be the set of centers
  of the circles of the given curve $\gamma $ and $C_{\Psi_0} $ be the centers
  of $\Psi_0$ (including centers of parts of circles).

 We  may assume that $\pm v\not\in C_\gamma$ (otherwise apply a
 rotation to $\gamma $).

 Since $C_\gamma$ and $C_{\Psi_0}$ are both 
  finite and $\pm v\not\in C_\gamma$, Corollary~\ref{lemma:easy 1}
  yields an 
  orthogonal matrix $U \in \cO (d)$ such that $Uv=v$ and 
  $$
  U C_\gamma \cap C_{\Psi_0} = \emptyset \, .
  $$

  The curve $\gamma '' = U \gamma $ consists of circles whose centers
  are disjoint from those of $\Psi_0 $, consequently $\gamma ''$ and
  $\Psi_0$ have only finitely many points in common. After mapping via
  $\phi _{x,r}$ we obtain a trajectory $\Gamma ''_{x,r}=\phi _{x,r}
  \circ \gamma ''$ in $\S ^d$
  consisting of circles, such that $w \in \Gamma '' _{x,r}$ and
  $\Gamma _{x,r}'' \cap \Psi $ is finite. 
  \end{proof}

  \begin{lemma}
    \label{topo}
  Let $\gamma$ be a closed curve in $\S^d$ with trajectory $\Gamma$, so that its covering radius satisfies 
  $r< \pi /4$. Then $\Gamma$ intersects the boundary $\partial  B_\sigma
  (x)$ for all $x\in \sdd$ and $\sigma \in ( \pi /4, \pi /2)$:
  $$
  \Gamma \cap \partial  B_\sigma
  (x) \neq \emptyset \, .
  $$
    \end{lemma}

    \begin{proof}
      Since the covering radius of $\Gamma $ is $r< \pi /4$, there
      exists a point $z\in \Gamma $, such that $\dist(z,x) \leq
      r$ and thus $z\in B_r(x) \subseteq B_\sigma (x)$. Similarly,
      there exists a point $\tilde{z}\in \Gamma $, such that $\dist(\tilde z,-x) \leq
      r$ and thus $\tilde z\in B_r(-x) \subseteq B_\sigma (-x)$. Since
      $$ \pi  = \dist(x,-x) \leq \dist(x,\tilde{z}) +
      \dist(\tilde{z},-x) \, ,
      $$
      we see that $\dist(\tilde{z},x) \geq \pi - r \geq
      \sigma$ and $\tilde{z} \not \in B_\sigma (x)$. Consequently, the 
      continuous function $\psi (s) = \dist(\gamma (s),x) $
      takes values $<\sigma $ and $>\sigma $. As a consequence, there exists
      $s_0$, such that $\psi (s_0) =    \dist(\gamma (s_0),x) =
\sigma $. In other words,  the point $\gamma (s_0) $ is in $\partial B_\sigma
      (x)$ or $\Gamma \cap  \partial B_\sigma
      (x)\neq \emptyset  $.
    \end{proof}

    By combining Lemma~\ref{crux2} with Lemma~\ref{lemma:Samko analog}, we
    obtain that 
\begin{equation}
  \label{eq:c11}
  \frac{1}{\ell(\gamma_{x,r})} \int _{\gamma _{x,r}} f = \int
  _{\partial B_r(x)} f = c_{d,k}(r) f(x) \qquad \text{ for all } f\in \cH
  _k ^d \, .
\end{equation}
As in Section~5 this formula paves the way to make a transition from
$t$-design points to $t$-design curves. 

\subsection{Part I of the proof of Theorem \ref{thm:Sd existence} }
\begin{proof}
 We first prove the existence of a cycle (a formal sum of closed, piecewise smooth
 curves) in $\S ^d$ that yields exact integration for $\Pi _t$.  
We prove this  claim by induction on the dimension $d$. The case $d=2$ corresponds to
Theorem \ref{thm:t-design} and shows that the optimal design curve  can be
realized as  a  union of circles. We now assume that the claim holds
for $d-1$ with unions of circles.

For $0<r<\frac{\pi}{2}$ and $x\in\S^d$, the induction hypothesis
combined  with Lemma \ref{crux2}(i)  shows that there is a sequence of $t$-design curves $\left(\gamma_{x,r,t}\right)_{t\in\N}$ for $\partial B_r(x)$ of length 
\begin{equation}\label{eq:length induction G}
	\ell(\gamma_{x,r,t})\lesssim \sin r \, t^{\frac{(d-1)(d-2)}{2}} \,,
\end{equation}
whose trajectories $\Gamma_{x,r,t}=\gamma_{x,r,t}([0,1])$ consist of unions of circles.

As in the proof of Theorem \ref{thm:t-design}, we use a sequence $(X_t)_{t\in\N}$ of asymptotically
optimal $t$-design points in $\S^d$ and verify that the cycle $\gamma
_t^r = \dotplus _{x\in X_t} \gamma _{x,r,t} $ associated to the
trajectory  $\Gamma^r_{t}:=\bigcup_{x\in X_t} \Gamma_{x,r,t}$ provides
an exact quadrature on $\Pi _t$. 
A careful choice of the radius $r$ will then yield a single closed
curve instead of a cycle. 

For the component $\mathcal{H}^2_0=\spann\{1\}$, we observe 
\begin{equation*}
\frac{1}{\ell(\gamma^r_t)}\int_{\gamma^r_t} 1 = 1 = \int_{\S^d} 1.
\end{equation*}
Next, we consider $f\in\mathcal{H}^d_k$  for $1\leq k\leq t$. On the
one hand, since $\ell (\gamma _r^t) = |X_t| \ell (\gamma _{x,r,t})$, Lemma \ref{crux2}(i) yields
\begin{align*}
\frac{1}{\ell(\gamma^r_t)}\int_{\gamma^r_t} f & = \frac{1}{|X_t|} \sum_{x\in X_t}\frac{1}{\ell(\gamma_{x,r,t})}\int_{\gamma_{x,r,t}} f\\
& =  \frac{1}{|X_t|} \sum_{x\in X_t}\int_{\partial B_{r}(x)}f.
\end{align*}
On the other hand, by Lemma~\ref{lemma:Samko analog} for $f\in \cH _k^d$, we have
$$
\frac{1}{|X_t|} \sum_{x\in X_t} \int_{\partial B_{r}(x)}f = c_{d,k}(r)
\frac{1}{|X_t|} \sum_{x\in X_t} f(x) = c_{d,k}(r) \int _{\S ^d} f  =0=\int _{\S ^d} f \,.
$$
In combination we obtain 
 \begin{equation*}
 \frac{1}{\ell(\gamma^r_t)}\int_{\gamma^r_t} f  = \int_{\S^d}  f 
\end{equation*}
for all $f\in \bigoplus _{k=0}^t \cH _k^d$. 
 This concludes the first part of the proof.
\end{proof}

\subsection{Part II of the proof of Theorem \ref{thm:Sd existence}:
  Existence of a single closed curve }
If $r$ is too small, then the trajectory $\Gamma^r_t$ is not
connected. If $r$ is too big, then we may not match the desired   asymptotics
$\ell(\Gamma^r_t)\lesssim t^{\frac{d(d-1)}{2}}$. Since 
$|X_t|\asymp t^d$ and the induction hypothesis yields $
\ell(\gamma_{x,r,t})\lesssim   t^{\frac{(d-1)(d-2)}{2}} \sin r   $,  
the total length
of  $\gamma^r_t$ is  
\begin{equation} \label{length0}
	\ell(\gamma^r_t)=\ell(\gamma_{x,r,t}) |X_t|  \asymp  \sin r \, 
        t^{\frac{(d-1)(d-2)}{2}} t^d \asymp   t^{\frac{d(d-1)}{2}+1} \sin r\,.
      \end{equation}
      This estimate suggests that we choose $r=r_t$ as $r_t \asymp t ^{-1}$.
Precisely, let  $\rho_t$ be the covering radius  of $X_t$, then we set 
\begin{equation*}
r_t:=2\rho_t,
\end{equation*}
 Therefore 
$\sin r\asymp \rho_t \lesssim t^{-1}$, so that \eqref{length0} leads
to the expected  total length
\begin{equation*}
	\ell(\gamma^{2\rho_t}_t) \lesssim  t^{-1}  t^{\frac{d(d-1)}{2}+1} \asymp t^{\frac{d(d-1)}{2}}.
\end{equation*}
To ensure that $\Gamma^{2\rho_t}_t$ is connected, we will construct
$\gamma_{x,2\rho_t,t}$, for $x\in X_t$, in a sequential fashion.  

\begin{proof}[Proof of Theorem \ref{thm:Sd existence} (Part II)]
As in Section~\ref{sec:sphere} we consider the  graph
$\mathcal{G}_{\rho_t}$ with vertices $X_t$ and edges between distinct
$x,y\in X_t$ if and only if $B_{\rho_t}(x)\cap B_{\rho_t}(y) \neq \emptyset$.  According to  Lemma \ref{lemma:G connection 0}, the graph $\mathcal{G}_{\rho_t}$ is connected. Therefore, it possesses a  spanning
tree 
$\mathcal{T}_{\rho_t}$~\cite{diestel}; this is a subgraph that contains all
vertices of $\cG _{\rho _t}$ such that every vertex $y$ can be reached
by a unique path from a root $x_0$. 

We start at the  root $x_0$ of $\mathcal{T}_{\rho_t}$ and  take a $t$-design curve $\gamma_{x_0,2\rho_t,t}$ for $\partial
B_{2\rho}(x_0)$. Recall from Lemma~\ref{crux2}
that $\gamma_{x_0,2\rho_t,t}$ is obtained as
the image of a $t$-design curve $\gamma $ in $\sdd$ via the diffeomorphism $\phi
_{x_0,2\rho _t}$ as $\gamma_{x_0,2\rho_t,t} = \phi
_{x_0,2\rho _t} (\gamma )$ and we suppose that the trajectory of $\gamma$ is a union of circles. 

Now consider the first descendant $x_1$ of $x_0$ in
the tree $\mathcal{T}_{\rho_t}$. Since $\dist(x_0,x_1)\leq
2\rho _t$, Lemma~\ref{lemma:induction curve sphere}(ii) (with $r=2\rho
_t$) implies that
$$
\partial B_{2\rho _{t}}(x_0) \cap
\partial B_{2\rho _t}(x_1) = \phi _{x_0,2\rho _t} \big( \partial B_\sigma (e_d) \big) 
$$
for some $\sigma > \pi/ 4$. Since $\gamma $ is a $t$-design curve in
$\sdd$, the covering radius of its trajectory is of the order $t ^{-1} $ and is thus smaller than
$\sigma $ for $t$ large enough.  Lemma~\ref{topo}  implies that
$\Gamma \cap \partial B_\sigma (e_d) \neq \emptyset $, and after
applying $\phi _{x_0, 2\rho _t}$ we obtain that
$$
\Gamma _{x_0,2\rho_t,t} \cap \partial B_{2\rho _{t}}(x_0) \cap
\partial B_{2\rho _t}(x_1)  \neq \emptyset \, .
$$
Let $w_1\in \S ^d$ be a point in this intersection.  
By Lemma \ref{crux2}(iv) applied to $\Psi = \Gamma _{x_0,2\rho _t,t}$
and $w_1$, there exists  a $t$-design curve
$\tilde \gamma$ in $\sdd $, whose image $\phi _{x_1,2\rho _t}
(\tilde{\gamma}) =  \gamma_{x_1,2\rho_t,t}$ is a $t$-design  for $\partial
B_{2\rho}(x_1)$ such that  $w_1\in \Gamma_{x_1,2\rho_t,t}$ and the
intersection $\Gamma_{x_0,2\rho_t,t}\cap\Gamma_{x_1,2\rho_t,t}$
contains only  finitely many points.

The next descendant $x_2$ leads to 
$$
\partial B_{2\rho _{t}}(x_1) \cap
\partial B_{2\rho _t}(x_2) = \phi _{x_1,2\rho _t} \big( \partial B_\sigma (e_d) \big) 
$$
for some $\sigma > \pi/ 4$. The same arguments as above yield the existence of  
$$
w_2\in \Gamma _{x_1,2\rho_t,t} \cap \partial B_{2\rho _{t}}(x_1) \cap
\partial B_{2\rho _t}(x_2)\,.
$$
We now apply Lemma \ref{crux2}(iv)  with $w_2$ and $\Psi = \Gamma
_{x_0,2\rho_t,t} \cup \Gamma _{x_1,2\rho_t,t}$ and obtain  a
$t$-design curve $\Gamma _{x_2,2\rho_t,t}$ in $\partial B_{\rho
  _t}(x_2)$, such that $w_2\in \Gamma _{x_2,2\rho_t,t}$ and $\Gamma
_{x_2,2\rho_t,t} \cap \big(\Gamma _{x_0,2\rho_t,t}\cup \Gamma
_{x_1,2\rho_t,t}\big)$ is finite.

This process is repeated till we reach a leaf of the spanning tree $\mathcal{T}_{\rho_t}$. 
Then we return to the last branch-off in $\mathcal{T}_{\rho_t}$ and
proceed with the next branch of the tree.

Finally, this construction leads to a connected set of circles in $\S
^d$, and the resulting  trajectory $\Gamma_t:=\bigcup_{x\in X_t}\Gamma_{x,2\rho_t,t}$ of all curves is a connected set with finitely many intersection points. 

By construction, $\Gamma_t$ is a connected set of finitely many
circles. Although Lemma \ref{lemma:G connection} is formulated for a
graph $\mathcal{G}$ constructed from finitely many circles in $\S^2$,
its proof only uses combinatorial arguments and hence also holds for
circles in $\S^d$. Since $\Gamma_t$ is connected, so is
$\mathcal{G}$. The second part of the proof of Lemma \ref{lemma:G
  connection} shows that we may fix an arbitrary orientation, and then the corresponding
directed graph $\mathcal{G}_{\rightarrow}$ is also connected. Thus,
$\Gamma_t$ can be traversed by a single continuous curve.  See again
Figure \ref{fig:sphere} for a pictorial argument.  
\end{proof}

\section{Some applications}\label{sec:Rd}
We now discuss a few direct applications of $t$-design curves to
mobile sampling on the sphere and exact integration of  polynomials with respect to the measure
$\mathrm{e}^{-\|x\|}\mathrm{d}x$ on $\R^d$.

\subsection{Mobile sampling on the sphere}
Here we prove Corollary~\ref{corintro} of the Introduction and show
that a polynomial $f$ of degree $t$ can be reconstructed from its
restriction to a $2t$-design curve.

For its formulation we recall that $\Pi _t$ restricted to $\S^d$ is a reproducing kernel
Hilbert space with respect to the inner product from $L^2(\S ^d)$. This means
that for every $x\in \S ^d$ there is a polynomial $k_x \in \Pi _t$,
such that
$$
f(x) = \langle f, k_x \rangle = \int _{\S ^d} f  \, \overline{k_x} \,
.
$$
This kernel possesses an explicit description by means of zonal
spherical harmonics and Gegenbauer (or ultraspherical) 
polynomials ~\cite{Stein:1971kx}. Let $P^\lambda _k,
k\in \mathbb{N},$
be the sequence of Gegenbauer polynomials associated with  $\lambda
>0$. They are   defined by their generating
function
$$
(1-2rx + r^2)^{-\lambda } = \sum _{k=0}^\infty P^\lambda _k(x) r^k \, .
$$
Using~\cite[Thm.~2.14]{Stein:1971kx}, there are real constants $b_{k,d}$,
such that
\begin{equation}
  \label{eq:rkrk}
  k_x (y) = \sum _{k=0}^t b_{k,d} P^{\frac{d-1}{2}}_k(x\cdot y) \,
  \qquad x,y \in \S ^d \, .
\end{equation}
We can now extend the formulation of Corollary ~\ref{corintro} of the Introduction as
follows.
\begin{proposition}
Let $\gamma $ be a $2t$-design curve on $\S ^d$ and $f$ a restriction of a polynomial
 of degree $t$ onto $\S^d$. Then
  \begin{equation}
   \label{eq:int4b}
   \frac{1}{\ell (\gamma )} \int _\gamma |f|^2 = \int _{\S ^d} |f|^2
   \, 
 \end{equation}
 and $f$ is reconstructed from its values along $\gamma$ by 
 \begin{equation}
   \label{eq:recon44}
   f(x)= \frac{1}{\ell (\gamma ) } \int _{\gamma} f k_x  ,\quad\text{for all } x\in\S^d\, .
 \end{equation}
\end{proposition}
\begin{proof}
  We only need to prove the reconstruction formula~\eqref{eq:recon44}. Since $f k_x\in \Pi _{2t}$ and $k_x$ is real-valued, the reproducing property yields
  \begin{equation*}
  f(x)  = \int _{\S ^d} f  \, k_x =  \frac{1}{\ell (\gamma ) } \int _{\gamma} f k_x\,.\qedhere
  \end{equation*}
\end{proof}

\subsection{Integration of polynomials on $\R^d$ with respect to $\mathrm{e}^{-\|x\|}\mathrm{d}x$}
Next we consider  the integration problem
\begin{equation*}
	\int_{\R^d} f(x)\mathrm{e}^{-\|x\|}\mathrm{d}x.
\end{equation*}
Recall that the family of generalized
Laguerre polynomials $(L^{(d-1)}_n)_{n\in\N}$ are orthogonal with
respect to the measure $r^{d-1}\mathrm{e}^{-r}\mathrm{d}r$ on
$[0,\infty)$. Using the zeros of $L_n^{(d-1)}$, we obtain the
following quadrature rule, where $\Pi_t$ now stands for polynomials of degree at most $t$ in $d$ variables. 
\begin{corollary}\label{corol:001}
	Let $\gamma$ be a spherical $t$-design curve in
        $\S^{d-1}$. For every integer $n\geq \frac{t+1}{2}$, 
	let $\{r_j\}_{j=1}^n\subset(0,\infty)$ be the set of zeros of
        $L^{(d-1)}_n$ with the  associated 
	weights $\{w_j\}_{j=1}^n\subset(0,\infty)$ for  Gaussian quadrature. Then we have
	\begin{equation*}
		\int_{\R^d} f(x)\mathrm{e}^{-\|x\|}\mathrm{d}x = \frac{\vol(\S^{d-1})}{\ell(\gamma)}\sum_{j=1}^{n} \frac{w_j}{r_j} \int_{r_j\gamma} f,
		\qquad \text{ for all } f\in\Pi_t.
	\end{equation*}
\end{corollary}
Thus, the scaled curves $\{r_j\gamma\}_{j=1}^{n}$ with weights $\{\frac{w_j\vol(\S^{d-1})}{r_j \ell(\gamma)}\}_{j=1}^{n}$ form an exact quadrature rule for $\Pi_t$ with respect to the measure $\mathrm{e}^{-\|x\|}\mathrm{d}x$. 
\begin{proof} 
Gaussian
quadrature based on the zeros of $L_n^{(d-1)}$ is exact for all univariate polynomials $g$ of degree at most $t\leq
2n-1$, i.e., 
\begin{equation*}
	\int_0^\infty g(r)r^{d-1}\mathrm{e}^{-r}\mathrm{d}r =
        \sum_{j=1}^{n} w_j g(r_j).
\end{equation*}
Let $f\in \Pi_t$. For fixed $z\in \S ^{d-1}$, the function
$r\mapsto f(rz)$ is a univariate polynomial of degree at most $t$ and therefore 
\begin{align*}
	\int_{\R^d} f(x)\mathrm{e}^{-\|x\|}\mathrm{d}x & = \int_{\S^{d-1}} \int_0^\infty f(rz)r^{d-1}\mathrm{e}^{-r}\mathrm{d}r
	\mathrm{d}z\\
	& =  \int_{\S^{d-1}} \Big(\sum_{j=1}^{n} w_j  f(r_j z ) \Big) \mathrm{d}z
	\intertext{Since $\gamma $ is a $t$-design curve on $\S
^{d-1}$ and $z\mapsto f(r_jz)$ is a polynomial in $\Pi_t$, we derive}
	\int_{\R^d} f(x)\mathrm{e}^{-\|x\|}\mathrm{d}x	& =  \frac{\vol(\S^{d-1})}{\ell(\gamma)}\int_\gamma  \sum_{j=1}^{n} w_j f(r_j\cdot)\\
	& =  \frac{\vol(\S^{d-1})}{\ell(\gamma)}\sum_{j=1}^{n} w_j \int_0^1 f(r_j\gamma(s))\|\dot{\gamma}(s)\|\mathrm{d}s\\
	& = \frac{\vol(\S^{d-1})}{\ell(\gamma)} \sum_{j=1}^{n} \frac{w_j}{r_j} \int_0^1 f(r_j\gamma(s))\|r_j\dot{\gamma}(s)\|\mathrm{d}s\\
	& = \frac{\vol(\S^{d-1})}{\ell(\gamma)}\sum_{j=1}^{n} \frac{w_j}{r_j} \int_{r_j\gamma} f.
\end{align*}
\end{proof}

\begin{example}\label{ex:001}
{\rm Let $\gamma^{(3,a_3)}$ be the spherical $3$-design curve of 
  Proposition \ref{prop:curve example}. The zeros of $L^{(2)}_2(r) =
  \tfrac{1}{2} r^2 -4r+6$ are $r_1=2$ and $r_2=6$. The associated Gaussian weights are $w_1=\frac{3}{2}$ and $w_2=\frac{1}{2}$. Therefore we obtain
\begin{equation}\label{eq:quad 3}
\int_{\R^3} f(x)\mathrm{e}^{-\|x\|}\mathrm{d}x = \frac{3\pi}{\ell(\gamma^{(3,a_3)})}\int_{2\gamma^{(3,a_3)}} f + \frac{\pi}{3\ell(\gamma^{(3,a_3)})}\int_{6\gamma^{(3,a_3)}} f,
\end{equation}
for $f\in\Pi_3$, see (a) in Figure \ref{fig:double}.}
\end{example}
\begin{figure}
\subfigure[The trajectories (red) $2\Gamma^{(3,a_3)}$ and (blue) $6\Gamma^{(3,a_3)}$ in Example \ref{ex:001}. ]{
\includegraphics[width=.38\textwidth]{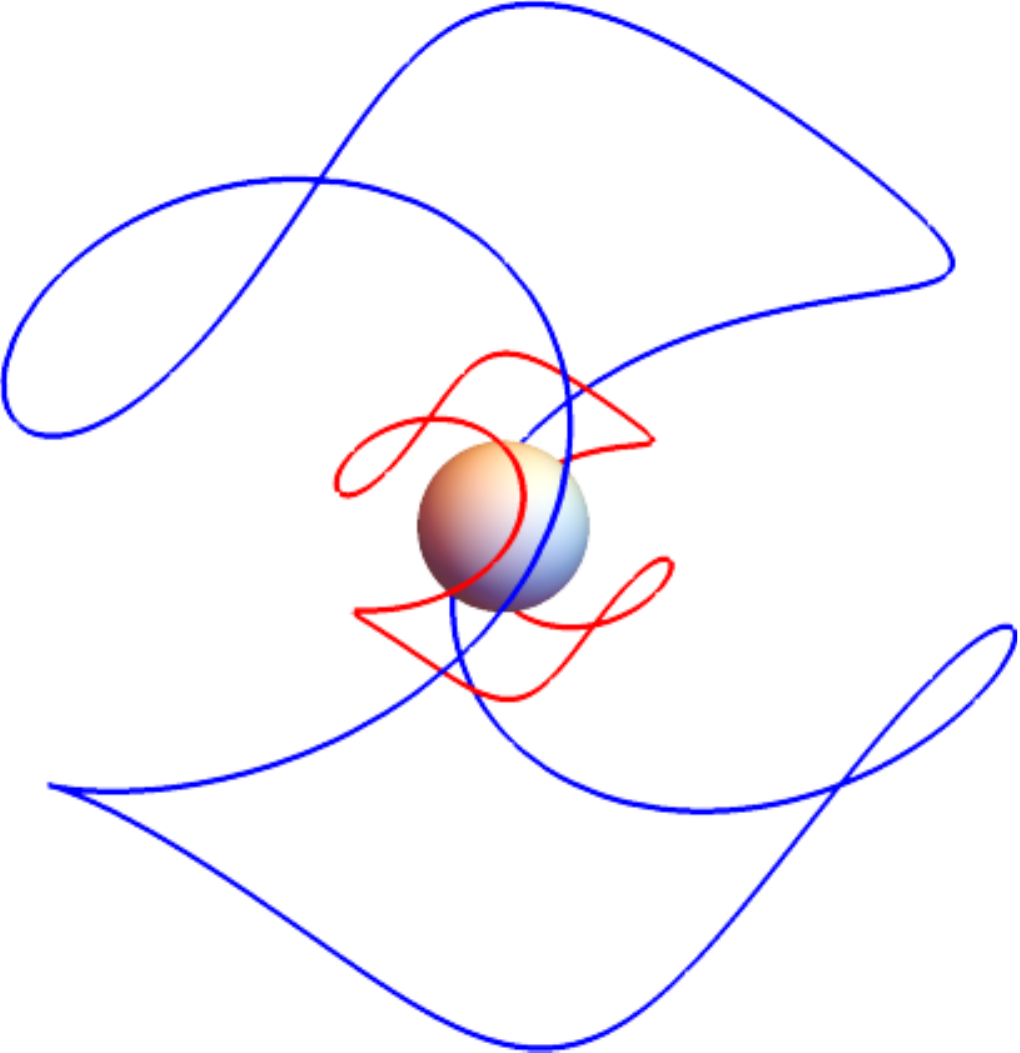}}\hspace{9ex}
\subfigure[Circles centered at $3$-design points in spheres of radius $2$ and $6$.]{
\includegraphics[width=.33\textwidth]{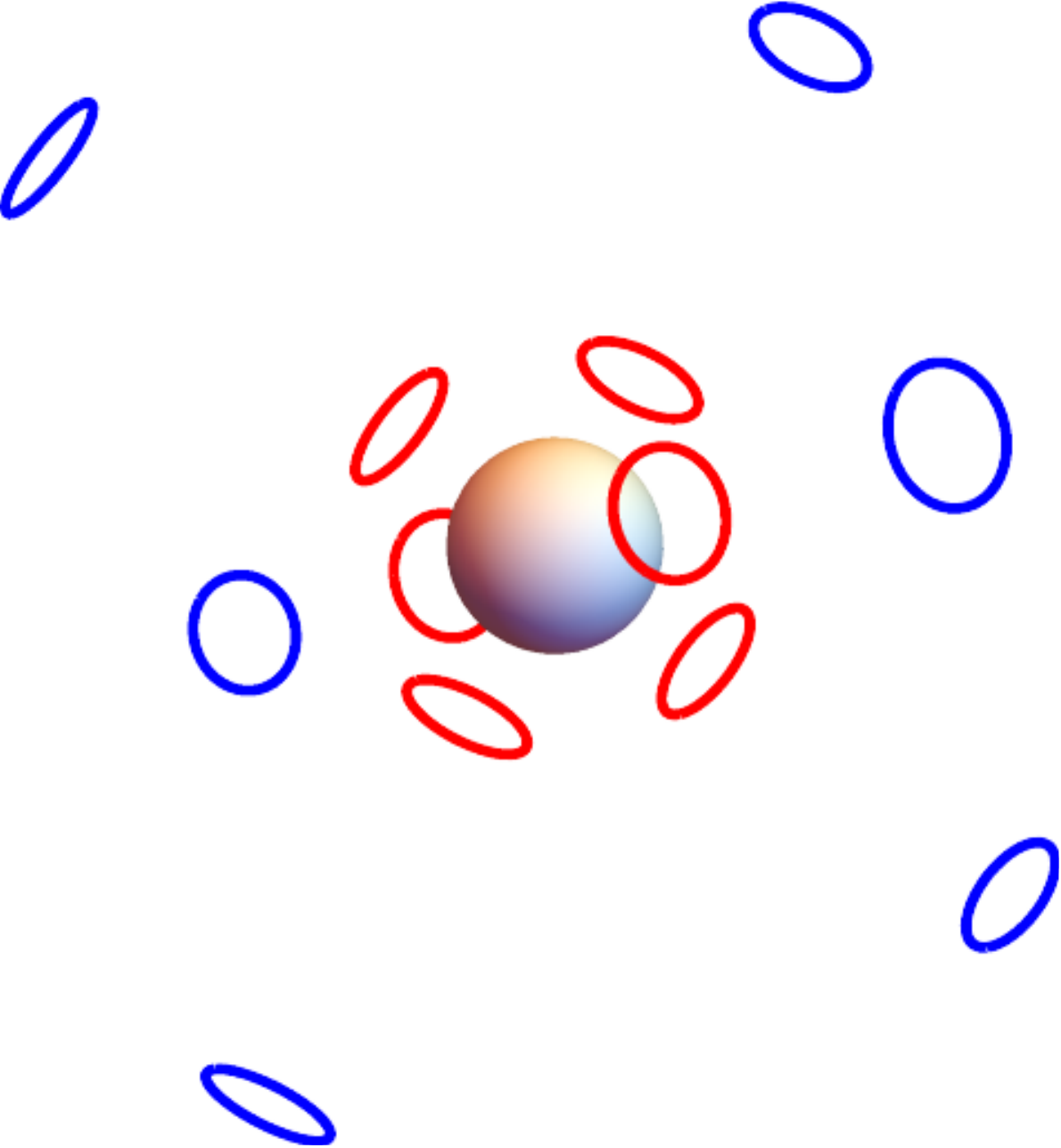}}
\caption{Trajectories provide exact integration of all polynomials in three variables of degree at most $3$ with respect to the measure $\mathrm{e}^{-\|x\|}\mathrm{d}x$. The unit sphere in the center is shown as a reference.}\label{fig:double}
\end{figure}

\begin{example}
{\rm  Consider $t$-design points $X_t\subset\S^{2}$ and take curves $\gamma_{x,r}$ whose trajectory are Euclidean circles $\Gamma_{x,r}$ of radius $\sin r$ centered at $x\in X_t$ as in the proof of Theorem \ref{thm:t-design}. Analogous to Corollary \ref{corol:001}, we may scale via the zeros of $L^{(d-1)}_n$ and use the associated Gaussian quadrature weights to obtain
\begin{equation*}
	\int_{\R^3} f(x)\mathrm{e}^{-\|x\|}\mathrm{d}x=\frac{2}{\sin r |X_t|}\sum_{j=1}^n \sum_{x\in X_t} \frac{w_j}{r_j  }\int_{r_j\gamma_{x,r}}  f,\qquad f\in\Pi_t,
\end{equation*}
see (b) in Figure \ref{fig:double} for $t=3$.}
\end{example}

\end{document}